\documentclass{amsart}

\usepackage{amssymb,amsmath,amsthm,amsfonts}
\usepackage{framed,comment,enumerate}
\usepackage[pdftex]{graphicx}
\usepackage{xcolor, framed,overpic}

\def\R {\mathbb{R}}
\def\C {\mathcal{C}}
\def\N {\mathbb{N}}
\def\S {\mathbb{S}}

\def\eps{\varepsilon}

\newcommand{\loc}{\mathrm{loc}}
\newcommand{\Lip}{\mathrm{Lip}}

\newcommand{\wc}{\rightharpoonup}
\newcommand{\de}[1] {\mathrm{d} #1}

\newcommand{\pa}{\partial}
\newcommand{\mf}[1]{\mathbf{#1}}

\newcommand{\mbfu}{\mathbf{u}}
\newcommand{\mbfv}{\mathbf{v}}

\newtheorem{proposition}{Proposition}[section]
\newtheorem{theorem}[proposition]{Theorem}
\newtheorem*{theorem*}{Theorem}
\newtheorem{corollary}[proposition]{Corollary}
\newtheorem{lemma}[proposition]{Lemma}

\theoremstyle{definition}
\newtheorem{definition}[proposition]{Definition}

\newtheorem{remark}[proposition]{Remark}
\numberwithin{equation}{section}

\title[Multidimensional entire solutions]{Multidimensional entire solutions \\ for an elliptic system modelling phase separation}

\author{Nicola Soave and Alessandro Zilio}



\address{
\hbox{\parbox{5.7in}{\medskip\noindent
 Nicola Soave\\
Mathematisches Institut, Justus-Liebig-Universit\"at Giessen, \\
Arndtstrasse 2, 35392 Giessen (Germany),\\[2pt]
{\em{E-mail address: }}{\tt nicola.soave@gmail.com, nicola.soave@math.uni-giessen.de.}\\[5pt]
Alessandro Zilio\\
Centre d'analyse et de math\'{e}matique sociales\\
\'{E}cole des Hautes \'{E}tudes en Sciences Sociales\\
190-198 Avenue de France, 75244, Paris CEDEX 13 (France) \\
{\em{E-mail address: }}{\tt azilio@ehess.fr, alessandro.zilio@polimi.it.}}}}

\begin{document}
\begin{abstract}
For the system of semilinear elliptic equations
\[
	\Delta V_i = V_i \sum_{j \neq i} V_j^2, \qquad V_i > 0	 \qquad \text{in $\R^N$}
\]
we devise a new method to construct entire solutions. The method 
%is sufficiently powerful to 
extends the existence results already available in the literature, which are concerned with the 2-dimensional case, also 
%while it can also be used 
in higher dimensions $N \ge 3$. In particular, we provide an explicit relation between orthogonal symmetry subgroups, optimal partition problems of the sphere, the existence of solutions and their asymptotic growth. This is achieved by means of new asymptotic estimates for competing system and new sharp versions for monotonicity formulae of Alt-Caffarelli-Friedman type.
\end{abstract}

\maketitle

\section{Introduction}

The elliptic systems
\begin{equation}\label{entire system}
\begin{cases}
\Delta V_i = V_i \sum_{j \neq i} V_j^2 \\
V_i \ge 0
\end{cases} \qquad \text{in $\R^N$, $i=1,\dots,k$},
\end{equation}
which arise in the blow-up analysis of phase-separation phenomena in coupled Schr\"odinger equations, has attracted an increasing attention in the last years, and by now many results concerning existence and qualitative properties of the solutions are available. For the detailed explanation about how \eqref{entire system} appears, we refer to \cite{BeLiWeZh,BeTeWaWe,SoZi2}. In this paper we prove the existence of \emph{$N$-dimensional solutions} to \eqref{entire system} in $\R^N$ for any $N \ge 2$. With this, we mean that we construct solutions in $\R^N$ which cannot be obtained from solutions in lower dimension by adding the dependence on some ``mute" variable. Our results extend the construction developed in \cite{BeTeWaWe}, which concerns the planar case $N=2$. In this perspective, we mention that previous results contained in \cite{BeLiWeZh,BeTeWaWe} only regard the existence of solutions in dimension $N=1$ or $2$, and the question of the existence in higher dimension was up to now open.

In order to state our main results, we introduce some notation. We denote by $\mathcal{O}(N)$ the orthogonal group of $\R^N$, and by $\mathfrak{S}_k$ the symmetric group of permutations of $\{1,\dots,k\}$. Let us assume that there exists a homomorphism $h: \mathcal{G} \to \mathfrak{S}_k$, where $\mathcal{G}< \mathcal{O}(N)$ is a nontrivial subgroup. We define the \emph{equivariant right action} of $\mathcal{G}$ on $H^1(\R^N,\R^k)$ in the following way:
\begin{equation}\label{equiv action}
\begin{split}
\mathcal{G} \times H^1(\R^N,\R^k) &\to  H^1(\R^N,\R^k) \\
(g,\mf{u}) &\mapsto g \cdot \mf{u}:= \left( u_{(h(g))^{-1}(1)} \circ g,\dots,u_{(h(g))^{-1}(k)} \circ g \right)
\end{split}
\end{equation}
where $\circ$ denotes the usual composition of functions, and we used the vector notation $\mf{u}:=(u_1,\dots,u_k)$. The set
\[
H_{(\mathcal{G},h)}:= \left\{ \mf{u} \in H^1(\R^{N},\R^k): \mf{u} = g \cdot \mf{u} \ \forall g \in \mathcal{G} \right\}
\]
is the subspace of the $(\mathcal{G},h)$-equivariant functions.
%By $\mathfrak{S}_k$ we denote the symmetric group of permutations of $\{1,\dots,k\}$; also $\mathfrak{S}_k$ acts on  $H^1(\R^N,\R^k)$, with action defined as follows:
%\begin{align*}
%\mathfrak{S}_k \times H^1(\R^N,\R^k) &\to  H^1(\R^N,\R^k) \\
%(\sigma,(\varphi_1,\dots, \varphi_k)) &\mapsto \sigma \cdot (\varphi_1,\dots, \varphi_k):= (\varphi_{\sigma(1)},\dots, \varphi_{\sigma(k)}).
%\end{align*}

\begin{definition}\label{def: admissible pair}
For $k \in \N$, a nontrivial subgroup $\mathcal{G} < \mathcal{O}(N)$, and a homomorphism $h: \mathcal{G} \to \mathfrak{S}_k$, we write that the triplet $(k,\mathcal{G},h)$ is admissible if there exists a $(\mathcal{G},h)$-equivariant function $\mf{u}$ with the following properties: 
\begin{itemize}
\item[($i$)] $u_i \ge 0$ and $u_i \not \equiv 0$ for every $i$;
\item[($ii$)] $u_i  u_j \equiv 0$ for every $i \neq j$;
\item[($iii$)] there exist $g_2,\dots,g_k \in \mathcal{G}$ such that
\[
u_2 = u_1 \circ g_2, \quad u_3 = u_1 \circ g_3, \quad \dots \quad u_k= u_1 \circ g_k. 
\]
%for $i=2,\dots,k$ there exists $g \in \mathcal{G}$ such that
%\[
%u_i = u_1 \circ g.
%\]
\end{itemize}
\end{definition}
\begin{remark}\label{rem: ne basta una}
Notice that, if $(k,\mathcal{G},h)$ is admissible triplet, then all the $(\mathcal{G},h)$-equivariant functions satisfy ($iii$) in the previous definition with the same symmetries $g_i$: indeed, by ($iii$) and equivariance we deduce that $(h(g_i))^{-1}(i) = 1$ for every $i$, so that any equivariant function satisfies
\begin{equation}\label{ne basta una eq}
v_i = v_{(h(g_i))^{-1}(i)} \circ g_i = v_1 \circ g_i, \qquad \forall i=1,\dots,k.
\end{equation}
This tells us that any equivariant function associated to an admissible triplet is completely determined by its first component: if we know that $\mf{v}$ is $(\mathcal{G},h)$-equivariant and that $(k,\mathcal{G},h)$ is an admissible triplet, then \eqref{ne basta una eq} holds true, and hence $v_2,\dots,v_k$ can be obtained by knowing $v_1$ and $g_2,\dots,g_k$.

We also underline the fact that there may exist symmetries in $\mathcal{G}$ whose corresponding permutation is the identity. In this case, these symmetries are imposed on the single components.

Finally, we observe that the definition of admissible triplet implicitly imposes several restrictions on $(k,\mathcal{G},h)$. For instance, by ($iii$) we immediately deduce that $h$ can never be the trivial homomorphism $g \in \mathcal{G} \mapsto id \in \mathfrak{S}_k$ for all $g$. Moreover, we also deduce that $\mathcal{G}$ has at least $k$ different elements.
\end{remark}
Let $(k,\mathcal{G},h)$ be an admissible triplet. We denote by 
\begin{equation}\label{def equiv functions on S}
\Lambda_{(\mathcal{G},h)}:= \left\{ \varphi \in H^1(\S^{N-1},\R^k)\left|\begin{array}{l} \text{$\varphi$ is the restriction on $\mathbb{S}^{N-1}$ of a }\\ \text{$(\mathcal{G},h)$-equivariant function fulfilling} \\
\text{($i$)-($iii$) in Definition \ref{def: admissible pair}} \end{array}\right.\right\}.
\end{equation}
We consider the minimization problem
\begin{equation}\label{def equivariant optimal value}
\ell_{(k,\mathcal{G},h)}:= \inf_{ \varphi \in \Lambda_{(\mathcal{G},h)} } \frac{1}{k}\sum_{i=1}^k \left( \sqrt{\left(\frac{N-2}{2}\right)^2 + \frac{\int_{\S^{n-1}} |\nabla_\theta \varphi_i|^2 }{\int_{\S^{n-1}} \varphi_i^2}} - \frac{N-2}{2}\right),
\end{equation}
where $\nabla_\theta$ denotes the tangential gradient on $\mathbb{S}^{N-1}$.

\begin{theorem}\label{thm: new existence}
For any admissible pair $(k, \mathcal{G},h)$, there exists a solution $\mf{V}$ of \eqref{entire system} with $k$ components in $\R^N$ satisfying the following properties:
\begin{itemize}
\item $\mf{V}$ is $(\mathcal{G},h)$-equivariant;
\item it results
\begin{equation}\label{growth rate strong}
\lim_{r \to +\infty} \frac{1}{r^{N-1+ 2\ell_{(k,\mathcal{G},h)}}} \int_{\pa B_r} \sum_{i=1}^k V_i^2 \in (0,+\infty).
\end{equation}
\end{itemize}
\end{theorem}
Here and in the rest of the paper $B_r(x_0)$ denotes the ball of center $x_0$ and radius $r$; in case $x_0=0$, we simply write $B_r$ for the sake of simplicity.

Since the theorem is quite general, we think that it is worth to spend some time making some explicit examples. This will be done in Section \ref{sec: app}. %in the last part of the introduction. 
For the moment, we anticipate that with our result we can both recover Theorem 1.3 and 1.6 in \cite{BeTeWaWe}, and moreover we can produce a wealth of new solutions existing only in dimension $N \ge 3$. 

We also observe that condition \eqref{growth rate strong} establishes that the solution $\mf{V}$ grows at infinity, in quadratic mean, like the power $|x|^{\ell_{(k,\mathcal{G},h)}}$. It is worth to remark that for any solution $\mf{V}$ to \eqref{entire system} it is possible to defined the \emph{growth rate} as the uniquely determined value $d \in (0,+\infty]$ such that
\[
\lim_{r \to +\infty} \frac{1}{r^{N-1+2m}} \int_{\pa B_r} \sum_{i=1}^k V_i^2 = \begin{cases} +\infty & \text{if $m <d$} \\
0  & \text{if $m>d$},\end{cases}
\]
see Proposition 1.5 in \cite{SoTe} and its proof. Therein, it is also shown that $\mf{V}$ \emph{has algebraic growth}, i.e. it satisfies the point-wise upper bound  
\begin{equation}\label{alg growth}
V_1(x)+ \cdots + V_k(x) \le C (1+|x|^\alpha) \qquad \forall x \in \R^N
\end{equation}
for some $C,\alpha \ge 1$, if and only if its growth rate $d$ is finite: we point out moreover that, as shown in \cite{SoZi1}, the system does indeed admit solutions with an exponential (i.e. non algebraic) growth.
%In such case, the minimal $\alpha$ for which there exists $C>0$ such that \eqref{alg growth} is satisfied is exactly $\alpha_{\min}=d$.   

 Theorem \ref{thm: new existence} not only specifies the growth rate of the function $(d = \ell(k,\mathcal{G},h))$, but also states that, for this precise growth rate, the limit 
 \[
 \lim_{r \to +\infty}  \frac{1}{r^{N-1+2d}} \int_{\pa B_r} \sum_{i=1}^k V_i^2
 \]
 is positive and finite. In this perspective we can prove that the solutions of Theorem \ref{thm: new existence} have minimal growth rate among all the possible $(\mathcal{G},h)$-equivariant solutions.

\begin{theorem}\label{prop: Liouville symm}
Let $(k, \mathcal{G},h)$ be an admissible pair, and let $\mf{V}$ be a $(\mathcal{G},h)$-equivariant solution of \eqref{entire system}. Then the growth rate of $\mf{V}$ is at least $\ell(k,\mathcal{G},h)$.
\end{theorem}

Both the proofs of Theorems \ref{thm: new existence} and \ref{prop: Liouville symm} exploit the hidden relationship between the elliptic system \eqref{entire system} and optimal partition problems of type \eqref{def equivariant optimal value}. This relationship arises for instance by means of the validity of the following modification of the celebrated Alt-Caffarelli-Friedman monotonicity formula, tailor made for the study of $(\mathcal{G},h)$-equivariant solutions. 

For $\mf{V} \in H^1(\R^N,\R^k)$ and $i=1,\dots,k$ we define
%\begin{equation}\label{acf functional factor}
\[
J_i(r):= \int_{B_r} \frac{|\nabla V_i|^2 + V_i^2 \sum_{j \neq i} V_j^2}{|x|^{N-2}}.
\]
%\end{equation}

\begin{proposition}\label{prop: acf equiv}
Let $(k,\mathcal{G},h)$ be an admissible triplet. There exists a constant $C>0$ depending only on $N$ and on $(k,\mathcal{G},h)$ such that, for any  $(\mathcal{G},h)$-equivariant solution $\mf{V}$ of \eqref{entire system}, the function
\[
r \mapsto \frac{1}{r^{2k \ell(k,\mathcal{G},h)}} e^{-Cr^{-1/2}} J_1(r) \cdots J_k(r)
\]
is monotone non-decreasing for $r > 1$; we recall that $\ell(k,\mathcal{G},h)$ has been defined in \eqref{def equivariant optimal value}. 
\end{proposition}
The expert reader will have already recognized the similarity with the original Alt-Caffarelli-Friedman monotonicity formula, proved in \cite{ACF}; monotonicity formulae of Alt-Caffarelli-Friedman type for competing systems are key ingredients for the results in \cite{ctv, FaSo, NoTaTeVe, SoTe, SoZi, Wa}. The previous result is, up to our knowledge, the first example of a monotonicity formula under a symmetry constraint.

%This means that the growth rate of the solutions of Theorem \ref{thm: new existence} is the minimal one in the class of the $(\mathcal{G},h)$-equivariant solutions. 

We review now the main known results regarding entire solutions of the system \eqref{entire system} which were already available, starting with the $k=2$ components system. The $1$-dimensional problem was studied in \cite{BeLiWeZh}, where it is proved that there exists a solution satisfying the symmetry property $V_2(x) = V_1(-x)$, the monotonicity condition $V_1'>0$ and $V_2'<0$ in $\R$, and having at most linear growth, in the sense that there exists $C>0$ such that
\[
V_1(x)+V_2(x) \le C(1+|x|) \qquad \forall x \in \R^N.
\] 
Up to translations, scaling, and exchange of the components, this is the unique solution in dimension $N=1$, see \cite[Theorem 1.1]{BeTeWaWe}. The linear growth is the minimal admissible growth for non-constant positive solutions of \eqref{entire system}. Indeed, in any dimension $N \ge 1$, if $(V_1,V_2)$ is a \emph{nonnegative} solution of \eqref{entire system} (which means that the condition $V_i >0$ is replaced by $V_i \ge 0$) and satisfies the sublinear growth condition
\[
V_1(x)+V_2(x) \le C(1+|x|^\alpha) \qquad \text{in $\R^N$}
\]
for some $\alpha \in (0,1)$ and $C>0$, then one between $V_1$ and $V_2$ is $0$, and the other has to be constant. This \emph{Liouville-type theorem} has been proved by B. Noris et al. in \cite[Propositions 2.6]{NoTaTeVe}.  

Differently from the problem in $\R$, in dimension $N = 2$, and hence in any dimension $N \ge 2$, system \eqref{entire system} with $k=2$ has infinitely many ``geometrically distinct" solutions, i.e. solutions which cannot be obtained one from the other by means of rigid motions, scalings, or exchange of the components, see \cite[Theorem 1.3]{BeTeWaWe} and \cite[Theorems 1.1 and 1.5]{SoZi1}. These solutions can be distinguished according to their growth rates and symmetry properties. In particular, in \cite{BeTeWaWe} the authors proved the existence of solutions having algebraic growth, while the results in \cite{SoZi1} concern solutions having exponential growth in $x$ and being periodic in $y$.

Regarding systems with several components, the aforementioned existence results admit analogue counterparts for any $k \ge 3$, see \cite[Theorem 1.6]{BeTeWaWe} and \cite[Theorem 1.8]{SoZi1}.

It is important to stress that the proofs in \cite{BeTeWaWe,SoZi1} use the fact that the problem is posed in dimension $N=2$, and apparently cannot be extended to higher dimension (see the forthcoming Remark \ref{rem: differenza noi loro} for a more detailed discussion).

In parallel to the existence results, great efforts have been devoted to the analysis of the $1$-dimensional symmetry of solutions under suitable assumptions; this, as explained in \cite{BeLiWeZh}, is inspired by some analogy in the derivation of \eqref{entire system} and of the Allen-Chan equation, for which symmetry results in the spirit of the celebrated De Giorgi's conjecture have been widely studied. In this context, we recall that assuming $k=2$ and $N=2$, A. Farina proved that if $(V_1,V_2)$ has algebraic growth and $\pa_2 V_1>0$ in $\R^2$, then $(V_1,V_2)$ is $1$-dimensional \cite{Fa}. In the higher dimensional case $N \ge 2$ with $k=2$, A. Farina and the first author proved a Gibbons-type conjecture for system \eqref{entire system}, see \cite{FaSo}. Furthermore, as product of the main results in \cite{Wa,Wa2}, K. Wang showed that any solution of \eqref{entire system} with $k=2$ having linear growth is $1$-dimensional. We mention also \cite[Theorem 1.8]{BeLiWeZh} and \cite[Theorem 1.12]{BeTeWaWe}, which are now included in the Wang's result.

As far as the $1$-dimensional symmetry for systems with $k >2$ is concerned, we refer to \cite[Theorem 1.3]{SoTe}, where the main results in \cite{FaSo,Wa,Wa2} are extended to systems with many components by means of improved Liouville-type theorems for multi-components systems, which put in relation the number of nontrivial components for a nonnegative solution of the first equation in \eqref{entire system} and its growth rate. In this perspective, Theorem \ref{prop: Liouville symm} is the counterpart of \cite[Theorem 1.7]{SoTe} in a $(\mathcal{G},h)$-equivariant setting. As a product of these two results, we can also derive the following corollary.

\begin{corollary}
For $k,N \in \N$, let 
\[
\mathcal{L}_k(\S^{N-1}) := \inf_{(\omega_1,\dots,\omega_k) \in \mathcal{P}_k} \sup_{i=1,\dots,k} \lambda_1(\omega_i),
\]
where $\mathcal{P}_k$ is the set of partitions of $\S^{N-1}$ in $k$ open disjoint and connected sets, and $\lambda_1$ denotes the first eigenvalue of the Laplace-Beltrami operator on $\S^{N-1}$. Let also $(k,\mathcal{G},h)$ be any admissible triplet, with $\mathcal{G}<\mathcal{O}(N)$. Then
\[
\mathcal{L}_k(\S^{N-1}) \le \ell(k,\mathcal{G},h).
\]
\end{corollary}

It is tempting to conjecture that equality holds for an appropriate choice of $(\mathcal{G},h)$, at least for some values of $k,N$. Indeed, in light of the known results in the literature, this is the case for $k=2$ and $k=3$, for every $N$. 
For $k=2$, the only (up to isometries) optimal partition for $\mathcal{L}_2(\S^{N-1})=1$ is the partition of the sphere in two equal spherical cups \cite{ACF}. This is clearly also an optimal partition for $\ell(2,\mathcal{G},h)$ if $\mathcal{G}$ is equal to the group generated by the reflection $T$ with respect to a hyperplane through the origin, and $h(T)$ is defined as the permutation exchanging the indices $1$ and $2$. In case $k=3$, an optimal partition for $\mathcal{L}_3(\S^{N-1})=3/2 (N-1/2)$ is the so-called $\mf{Y}$-partition (see \cite{HofHelTer, SoTe}) which is then optimal also for $\ell(3,\mathcal{G},h)$ if $\mathcal{G}$ is equal to the group generated by the rotation $R$ of angle $2\pi/3$ around the $x_N$ axis and $h(R)$ is the permutation mapping $1$ into $2$, $2$ into $3$ and $3$ into $1$.

\medskip

To conclude, we mention also the contribution \cite{WaWe}, where the authors considered the fractional analogue of \eqref{entire system}. Such problem exhibit new interesting phenomena with respect to the local case. Moreover, we observe that our results, as those in \cite{BeTeWaWe}, seem to be somehow connected with those in \cite{WeWethLincei}, which on the other hand concern finite energy decaying solutions of a different problem. 

\medskip

\textbf{Structure of the paper:} in Section \ref{sec: prel} we recall some known results needed for the rest of work, and which permits to show, in Subsection \ref{sec: app}, several concrete applications of Theorem \ref{thm: new existence}. Section \ref{sec: acf} is devoted to the proof of the equivariant Alt-Caffarelli-Friedman monotonicity formula, Proposition \ref{prop: acf equiv}; finally, in Section \ref{sec: main}, we give the proofs of the other main results, Theorem \ref{thm: new existence} and \ref{prop: Liouville symm}.

\section{Preliminaries and application of Theorem \ref{thm: new existence}}\label{sec: prel}

We introduce some notation and review some known results.
Let $\beta>0$, and let $\mf{U}$ be a solution to
\begin{equation}\label{syst beta in palla}
\begin{cases}
\Delta U_i = \beta U_i \sum_{j \neq i} U_j^2 & \text{in $B_R$} \\
U_i>0 & \text{in $B_R$}.
\end{cases}
\end{equation}
For $0<r<R$, we set
\[
\begin{split}
\bullet \quad & H(\mf{U},r):= \frac{1}{r^{N-1}} \int_{\partial B_r} \sum_{i=1}^k U_i^2\\
\bullet \quad & E(\mf{U},r):= \frac{1}{r^{N-2}} \int_{B_r} \sum_{i=1}^k |\nabla U_i|^2 + \beta\sum_{1 \le i<j\le k} U_i^2 U_j^2  \\
\bullet \quad & N(\mf{U},r):= \frac{E(\mf{U},r)}{H(\mf{U},r)} \quad \text{Almgren frequency function}.
\end{split} 
\]
Under the previous notation, by Proposition 5.2 in \cite{BeTeWaWe} it is known that $N(\mf{U},\cdot)$ is monotone non-decreasing for $0<r<R$,
\[
\frac{d}{dr}H(\mf{U},r)= \frac{2}{r} E(\mf{U},r) + \frac{2\beta}{r^{N-1}}\int_{B_r} \sum_{i<j} U_i^2 U_j^2 > 0,
\]
and for any such $r$
\begin{equation}\label{integrability del resto}
\int_{1}^r 2\beta\frac{\int_{B_s} \sum_{i<j} U_i^2 U_j^2}{s^{N-1} H(\mf{U},s)} \,ds \le N(\mf{U},r).
\end{equation}
The frequency function, also called Almgren's quotient, gives information about the behaviour of the solutions with respect to radial dilations. %In this perspective, the previous monotonicity formula tells us that solutions of \eqref{entire system} are radially increasing. 
Indeed, the possibility of defining a growth rate for any solution to \eqref{entire system} is a direct consequence of the monotonicity of $N(\mf{V},\cdot)$. We recall that, as proved in \cite[Proposition 1.5]{SoTe}, for any solution $\mf{V}$ to \eqref{entire system} there exists a value $d \in (0,+\infty]$ such that 
\begin{equation}\label{growth rate}
\lim_{r \to +\infty} \frac{\frac{1}{r^{N-1}} \int_{\pa B_r} \sum_{i=1}^k V_i^2}{r^{2d'}} = \begin{cases} +\infty & \text{if $d'<d$} \\
0 & \text{if $d'>d$},
\end{cases}
\end{equation}
and $d <+\infty$ if and only if $\mf{V}$ has algebraic growth. We write that $d$ is the \emph{growth rate of $\mf{V}$}, and it is remarkable that 
\begin{equation}\label{growth rate = almgren}
d= \lim_{r \to +\infty} N(\mf{V},r),
\end{equation}
see again \cite[Proposition 1.5]{SoTe} (the result is stated in \cite{SoTe} for solutions with algebraic growth, but its proof works also without such assumption). Notice that on the left hand side of \eqref{growth rate} we have the quadratic average of $\mf{V}$ on spheres of increasing radius divided by a power of $r^2$: thus the name \emph{growth rate}.

%Concerning the monotonicity of the Almgren frequency function, it is worth to mention that $N(\mf{U},\cdot)$ is in fact strictly increasing in $r$ for every $\mf{U}$ solving \eqref{syst beta B_1}. This is a consequence of the fact that $U_i>0$ in $B_R$ for every $i$ and from a well known characterization of functions with constant $N(\mf{U},\cdot)$, see Remark  

In the previous discussion $\beta>0$ was fixed. Let us now consider a sequence of parameters $\beta \to +\infty$, and a corresponding sequence $\{\mf{U}_\beta\}$ of solutions to \eqref{syst beta in palla}. The asymptotic behaviour of the family $\{\mf{U}_\beta\}$ has been studied in a number of papers \cite{BeLiWeZh,DaWaZh,NoTaTeVe,SoZi,SoZi2,tt,WeiWeth}, and many results are available. We only recall that, if the sequence is bounded in $L^\infty(B_R)$, then it is in turn uniformly bounded in $\Lip(B_R)$, and hence up to a subsequence it converges to a limit $\mf{U}$ in $\mathcal{C}^{0,\alpha}(B_R)$ and in $H^1_{\loc}(B_R)$ (see \cite{SoZi,NoTaTeVe}). If $\mf{U} \not \equiv \mf{0}$, then $\mf{U}$ is Lipschitz continuous and $\{\mf{U}=\mf{0}\}$ has Hausdorff dimension $N-1$. Moreover, $H(\mf{U},r)$ is non-decreasing and is $\neq 0$ for every $r>0$  (see \cite{tt}).

An important application to this asymptotic theory stays in the possibility of defining blow-down limits of entire solutions to \eqref{entire system}. We recall part of \cite[Theorem 1.4]{BeTeWaWe} ($k=2$) and \cite[Theorem 1.4]{SoTe} ($k$ arbitrary). Let $\mf{V}$ be a solution to \eqref{entire system}, and for any $R>0$ let us define the \emph{blow-down family}
\[
\mf{V}_R(x):= \frac{1}{H(\mf{V},R)^{1/2}} \mf{V}(Rx).
\]
If $\mf{V}$ has algebraic growth, i.e. its growth rate $d = N(\mf{V},+\infty)$ is finite, then $\{\mf{V}_R\}$ converges, in $\mathcal{C}^{0,\alpha}_{\loc}(\R^N)$ and in $H^1_{\loc}(\R^N)$,  as $R \to +\infty$ and up to a subsequence, to a homogeneous vector valued function $\mathbf{V}_\infty$ with homogeneity degree $d$ and such that
\begin{itemize}
\item the components $V_{i,\infty}$ are nonnegative and with disjoint support: $V_{i,\infty} V_{j,\infty} \equiv 0$ for every $i \neq j$;
\item for any $i \neq j$, $V_{i,\infty}-V_{j,\infty}$ is harmonic in the interior of its support.
\end{itemize}
In case $k=2$, it results then that $(V_{1,\infty},V_{2,\infty})=(\Psi^+,\Psi^-)$, where $\Psi$ is a homogenous harmonic polynomial in $\R^N$, and hence necessarily $d$ is an integer number.

\subsection{A wealth of new solutions: applications of Theorem \ref{thm: new existence}}\label{sec: app}
We recalled that, for any $k \ge 2$, problem \eqref{entire system} has several solutions in $\R^2$. Clearly, these are also solutions in higher dimension, and up to now it was an open question whether or not there exist $N$-dimensional solutions of \eqref{entire system} in $\R^N$ with $N \ge 3$, i.e. solutions in $\R^N$ which cannot be obtained as solutions in $\R^{N-1}$ by adding the dependence of a variable. Theorem \ref{thm: new existence} gives a positive answer to these questions. In what follows we show how to use Theorem \ref{thm: new existence} as a recipe to construct entire solutions of \eqref{entire system}.

%several applications of Theorem \ref{thm: new existence}, showing that this result can be used as a recipe to construct entire solutions of \eqref{entire system}.
\medskip

\noindent \textbf{A concrete example in $\R^3$ for $k=2$.} To start with a very concrete example, we focus on problem \eqref{entire system} in $\R^3$ with $k=2$, and we examine the case where $\mathcal{G}$ is equal to the group of symmetries generated by the reflections $T_1,T_2,T_3$ with respect to the planes $\{x=0\}$, $\{y=0\}$, and $\{z=0\}$ respectively, and $h:\mathcal{G} \to \mathfrak{S}_k$ is defined on the generators of $\mathcal{G}$ by $h(T_i)=(1 \ 2)$ for every $i$. We used here the standard notation $(1 \ 2)$ to denote the cycle mapping $1$ in $2$, and $2$ in $1$. In order to check that this is an admissible triplet, we verify that
\[
(u_1,u_2) = \left( (xyz)^+, (xyz)^- \right)
\]
is a $(\mathcal{G},h)$-equivariant function satisfying ($i$)-($iii$) in Definition \ref{def: admissible pair}. For the equivariance, we explicitly observe that
\[
T_i \cdot (u_1,u_2) = (\text{see \eqref{equiv action}}) = (u_2 \circ T_i, u_1 \circ T_i) = (\text{def. $\mf{u}$}) = (u_1,u_2),
\]
for every $i$, and since $\mathcal{G}$ is generated by $T_1,T_2,T_3$, this is sufficient to conclude that $\mf{u}$ is $(\mathcal{G},h)$-equivariant. Points ($i$) and ($ii$) in Definition \ref{def: admissible pair} are staightforward, and ($iii$) is satisfied since $u_2 = u_1 \circ T_i$ for any $i$. As a consequence, by Theorem \ref{thm: new existence} there exists a $(\mathcal{G},h)$-equivariant solution $(V_1,V_2)$ of \eqref{entire system} in $\R^3$ with $k=2$, having growth rate equal to $\ell(k,\mathcal{G},h)= N(\mf{V},+\infty)$ (we recall that the growth rate is always equal to the limit at infinity of the Almgren frequency function, see \eqref{growth rate = almgren}). Since the symmetries of $\mathcal{G}$ involve the $3$ variables, this solution cannot be obtained by a $2$-dimensional solution adding the dependence of $1$-variable: $V_1-V_2$ is not constant since $\mf{V}$ has growth rate $\ell(2,\mathcal{G},h)>0$; moreover, thanks to the symmetries $T_1,T_2,T_3$, we have that the function $V_1-V_2$ vanishes on the set $\{x =0 \} \cup \{y = 0\} \cup \{z = 0\}$. Since the projection of this set on any two-dimensional subspace is equal to the entire subspace but $\mathbf{V}$ is non trivial, we immediately deduce that the solution can not be two dimensional.

In this particular case we can also explicitly compute $\ell(2,\mathcal{G},h)$, in the following way: by minimality
\[
\ell(2,\mathcal{G},h) \le \frac{1}{2}\left( \sqrt{\frac{1}{4} + \frac{\int_{\mathbb{S}^2} |\nabla_\theta(xyz)^+|^2}{\int_{\S^2} |(xyz)^+|^2}} -\frac{1}{2}\right) + \frac{1}{2} \left(\sqrt{\frac{1}{4} + \frac{\int_{\S^{2}} |\nabla_\theta(xyz)^-|^2}{\int_{\S^{2}} |(xyz)^-|^2}} -\frac{1}{2} \right),
\]
and the right hand side is equal to $3$: indeed, since $\Phi:= xyz$ is a homogeneous harmonic polynomial of degree $3$, its angular part $\Phi|_{\S^2}$ solves
\[
-\Delta_\theta \Phi|_{\S^2}= 12 \Phi|_{\S^2} \qquad \text{in $\S^2$},
\]
and this permits to carry on explicit computations. This means that $\Psi$ (the blow-down limit) is a homogeneous harmonic polynomial of degree $\ell(2,\mathcal{G},h) \le 3$. It is then necessary that $\Psi=\Phi=xyz$: to check this, we can simply consider all the homogeneous harmonic polynomials in $\R^3$ with degree $\le 3$, which are classified, and observe that the only one being $(\mathcal{G},h)$ equivariant is $\Phi$. As a consequence, the degree of homogeneity of $\Psi$ is $3=\ell(2,\mathcal{G},h)$.

\medskip

\noindent \textbf{General case in $\R^N$ with $k=2$.} The very same argument as before can be considered by taking any homogeneous harmonic polynomial $\Phi$ in $\R^N$ of degree $d \in \N$, with a nontrivial finite group of symmetry $\mathcal{G}$: with this we mean that there exists a group of symmetry with generators $T_1,\dots,T_m$ such that $\Phi^{\pm} \circ T_i = \Phi^{\mp}$. To any $T_i$ we associate the cycle $(1 \ 2)$. This induces a homomorphism $h: \mathcal{G} \to \mathfrak{S}_2$, and it is not difficult to check that $(2,\mathcal{G},h)$ is an admissible triplet. Indeed, by assumption the pair $(u_1,u_2)=(\Phi^+,\Phi^-)$ fulfills ($i$)-($iii$) in Definition \ref{def: admissible pair}, and is $(\mathcal{G},h)$-equivariant:  the equivariance follows by
\[
T_i \cdot (u_1,u_2) = (\text{see \eqref{equiv action}}) = (u_2 \circ T_i, u_1 \circ T_i) = (u_1,u_2)
\] 
for any $i$. Points ($i$) and ($ii$) in Definition \ref{def: admissible pair} are trivial, and ($iii$) is satisfied since $u_2=u_1 \circ T_i$ for any $i$ by assumption. If, as in the example above, the group $\mathcal{G}$ has been chosen from the beginning so that the symmetries of $\mathcal{G}$ involve all the $N$-variables, we obtain an $N$-dimensional solution to \eqref{entire system}.
% are necessarily $N$-dimensional, then in particular $\mf{V}$ cannot be obtained by solutions in $\R^{N-1}$ adding the dependence of one or more variable. 
Explicit cases where the previous argument is applicable are the following: 
\begin{itemize}
\item At first, we show how we can recover Theorem 1.3 in \cite{BeTeWaWe}. In dimension $N=2$, we take $\Phi_d(x,y):= \mathfrak{Re}((x+iy)^d)$, with $d \in \N$. Then $\Phi_d$ is symmetric, in the previous sense, with respect to the group of symmetry generated by the reflections $T_1,\dots,T_d$ with respect to its nodal lines: $\Phi_d^{\pm} \circ T_i = \Phi_d^{\mp}$. By the previous argument, we find $(\mathcal{G},h)$-equivariant solutions of the problem with growth rate $\ell(2,\mathcal{G},h)$, which clearly are $2$-dimensional. Reasoning as in pur first example, it is not difficult in this case to check that $\ell(2,\mathcal{G},h)=d$.
%Being $(\Phi_d^+,\Phi_d^-)$ a $(\mathcal{G},h)$-equivariant function, we have
%\[
%\ell(2,\mathcal{G},h) \le \frac{1}{2}\sqrt{\frac{\int_{\mathbb{S}^2} |\nabla_\theta\Phi_d^+|^2}{\int_{\S^2} |\Phi_d^+|^2}} + \frac{1}{2} \sqrt{\frac{\int_{\S^{2}} |\nabla_\theta\Phi_d^-|^2}{\int_{\S^{2}} |\Phi_d^-|^2}},
%\]
%and the right hand side is equal to $d$ (recall that, since $\Phi_d$ is a homogeneous harmonic polynomial of degree $d$, its angular part solves
%\[
%-\Delta_\theta \Phi_d|_{\S^2}= d^2 \Phi_d|_{\S^2} \qquad \text{in $\S^2$}).
%\]
%This means that $\ell(2,\mathcal{G},h) \le d$. On the other the blow-down sequence $\mf{V}_R$ is convergent in $\mathcal{C}^{0,\alpha}_{\loc}(\R^N)$, as $R \to \infty$ and up to a subsequence, to a $(\mathcal{G},h)$-equivariant homogeneous harmonic polynomial $\Psi$ of degree $\ell(2,\mathcal{G},h)= N(\mf{V},+\infty)$. It is then necessary that $\Psi=\Phi_d$: for this, we can consider all the homogeneous harmonic polynomials in $\R^2$ with degree $\le d$, which are classified, and observe that the only one being $(\mathcal{G},h)$ equivariant is $\Phi_d$. As a consequence, the degree of homogeneity of $\Psi$ is $d=\ell(2,\mathcal{G},h)$.}
\item Secondly, we construct infinitely many new solutions in $\R^3$. We take $\Phi_d(x,y):= \mathfrak{Re}((x+iy)^d)z$, with $d \in \N$. Let $T_1,\dots,T_d$ denote the reflections with respect to the nodal planes of $\mathfrak{Re}((x+iy)^d)$, and let $T_z$ denote the reflection with respect to $\{z=0\}$. Then $\Phi_d^\pm \circ T_i = \Phi_d^{\mp}$, so that the general argument above is applicable, and hence we find a $(\mathcal{G},h)$-equivariant solution of \eqref{entire system} with growth rate $\ell(2,\mathcal{G},h)$. As in the first example, since the nodal set of $V_1-V_2$ has surjective projection on any $2$-dimensional subspace, $\mathbf{V}$ is necessarily $3$-dimensional. We can also check that $\ell(2,\mathcal{G},h)=d+1$. Being $(\Phi_d^+,\Phi_d^-)$ a $(\mathcal{G},h)$-equivariant function, we have 
\begin{align*}
\ell(2,\mathcal{G},h) & \le \frac{1}{2}\left(\sqrt{\frac{1}{4} +\frac{\int_{\mathbb{S}^2} |\nabla_\theta\Phi_d^+|^2}{\int_{\S^2} |\Phi_d^+|^2}} - \frac{1}{2}\right) \\
& \quad + \frac{1}{2} \left(\sqrt{\frac{1}{4}+\frac{\int_{\mathbb{S}^2} |\nabla_\theta\Phi_d^-|^2}{\int_{\S^2} |\Phi_d^-|^2}} - \frac{1}{2}\right).
\end{align*}
as in the previous example, we can prove that the right hand side is equal to $d+1$. On the other hand, using the blow-down theorem and explicitly observing that the only $(\mathcal{G},h)$-equivariant homogeneous harmonic polynomial in $\R^3$ with degree less than or equal to $d+1$ is $\Phi_d$, we conclude that $\ell(2,\mathcal{G},h)=d+1$.
\item We conclude with %\textcolor{gray}{a last concrete example in $\R^4$. Take $\Phi(x) = x_1 x_2 x_3 x_4$, and let $\mathcal{G}$ be the group of symmetries generated by the reflections $T_1,\dots,T_4$ with respect to the coordinate planes $\{x_i=0\}$, $i=1,\dots,4$. Then $\Phi^{\pm} \circ T_i = \Phi^{\mp}$ for every $i$. Hence, by the general argument above we obtain a $(\mathcal{G},h)$-equivariant solution $\mf{V}$ to \eqref{entire system} having growth rate $\ell(2,\mathcal{G},h)$. Following the line of reasoning of the previous examples, it is not difficult to check that $\mf{V}$ is $4$-dimensional and that $\ell(2,\mathcal{G},h)=4$.} 
the observation that the previous constructions can be extended in any dimensions. For instance we can consider the harmonic polynomial $\Phi = x_1 \cdots x_N$, together with the symmetry group generated by the reflections $T_1,\dots,T_N$ with respect to the coordinate planes $\{x_i=0\}$, $i=1,\dots,N$; notice that $\Phi^\pm \circ T_i = \Phi^\mp$ for any $i$. In the same way we could consider the harmonic polynomial $\Psi = \mathfrak{Re}((x_1+ix_2)^d) x_3 \cdots x_N$, together with symmetry group generated by the reflections $T_1,\dots,T_d$ with respect to the nodal hyperplanes of $\mathfrak{Re}((x_1+ix_2)^d)$, and by $R_3, \dots, R_N$, reflections with respect to the coordinate planes $\{x_i=0\}$, $i=3,\dots,N$.
\end{itemize}

\medskip 

\noindent \textbf{The case $k \ge 3$ in $\R^2$.} %The situation is much more delicate. 
For $k \ge 3$ components, we first show how to recover Theorem 1.6 in \cite{BeTeWaWe}. We focus then for the moment on the dimension $N=2$. Let $k \ge 3$, and for any $m \in \N$ let $d=mk/2$. We denote by $R_d$ the rotation of angle $\pi/d$, by $T_y$ the reflection with respect to $\{y=0\}$ (this corresponds to consider complex conjugation in $\mathbb{C}$), and we consider the group $\mathcal{G}<\mathcal{O}(N)$ generated by $R_d$ and $T_y$. We define a homomorphism $h:\mathcal{G} \to \mathfrak{S}_k$ (the group of permutations of $\{1,\dots,k\}$) letting
\[
h(R_d) := (1 \ 2 \ \cdots \ d) \quad \text{and} \quad h(T_y): i \mapsto k+2-i,
\]
where the indexes are counted modulus $k$. We can explicitly check that $(k,\mathcal{G},h)$ is an admissible triplet. Let us consider the function 
\begin{align*}
u_1 &:= \begin{cases} r^d \cos (d \theta) & \text{in $\bigcup_{i=0}^{m-1} R_d^{ik}(\{-\pi/2d<\theta<\pi/2d\})$} \\
0 & \text{otherwise}
\end{cases}\\
u_2 &:= u_1 \circ R_d \\
\vdots \\
u_k & : =u_{k-1} \circ R_d = u_1 \circ R_d^{k-1}.
\end{align*}
It is $(\mathcal{G},h)$-equivariant, as
\begin{align*}
R_d \cdot \mf{u} &= \left( u_k \circ R_d, u_1 \circ R_d, \dots, u_{k-1} \circ R_d\right) = \mf{u} \\
T_y \cdot \mf{u} & = \left( u_1 \circ T_y, u_k \circ T_y, u_{k-1} \circ T_y, \dots, u_3 \circ T_y, u_2 \circ T_y\right) = \mf{u}.
\end{align*}
It clearly satisfies ($i$) and ($ii$) in Definition \ref{def: admissible pair}, and for ($iii$) it is sufficient to note that $u_j = u_1 \circ R_d^{j-1}$ for every $j =2,\dots,k$. By Theorem \ref{thm: new existence}, we obtain a $(\mathcal{G},h)$-equivariant solution $\mf{V}$ of \eqref{entire system}; the fact that $\mf{V}$ is $2$-dimensional follows again from the symmetries: if $\mf{V}$ were $1$-dimensional, then we could say that $\cup_{i \neq j} \{V_i-V_j=0\}$ is the union of straight parallel lines. But on the other hand $\{V_2-V_3=0\}=R_d(\{V_1-V_2=0\})$, which cannot be parallel whenever $d >1$, i.e. whenever $k\ge 3$. 

To complete the analogy with the results in \cite{BeTeWaWe}, we still would have to prove that $N(\mf{V},+\infty)=\ell(k,\mathcal{G},h)$ is equal to $d$. Since we are in dimension $N=2$, this can be done by means of explicit computations, following the line of reasoning already adopted in the previous examples. We decided to not stress on this point for the sake of brevity.

%In order to show this, we observe that by minimality
%\[
%\ell(k,\mathcal{G},h) \le \frac{1}{k}\sum_{i=1}^k \sqrt{\left(\frac{N-2}{2}\right)^2 + \frac{\int_{\mathbb{S}^2} |\nabla_\theta u_i|^2}{\int_{\S^2} |u_i|^2}}-\left(\frac{N-2}{2}\right).
%\]
%Since $\mf{u}$ is explicitly given, we can check that the right hand side is equal to $d$, and hence $\ell(k,\mathcal{G},h) \le d$. On the other hand, we know that the blow-down family $\mf{V}_R$ is convergent in $\mathcal{C}^{0,\alpha}_{\loc}(\R^N)$, as $R \to \infty$ and up to a subsequence, to a $(\mathcal{G},h)$-equivariant homogeneous function of type
%\[
%\mf{V}_\infty= r^{\ell(2,\mathcal{G},h)}\left(g_1(\theta),\dots,g_k(\theta)\right).
%\]
%Notice that 
% $\ell(2,\mathcal{G},h)= N(\mf{V},+\infty)$.

\medskip

\noindent \textbf{The general case $k \ge 3$ in $\R^3$.} The case $k \ge 3$ and $N \geq 3$ is intrinsically more involved, and hence we focus on some particular examples given by the group of symmetry of the Platonic polyhedra. Let us consider for instance the group $\mathcal{G}_4< \mathcal{O}(N)$ associated to the tetrahedron $\mathcal{T}$. It is known that this group is isomorphic to $\mathfrak{S}_4$. The isomorphism $h_4$ is obtained labelling all the vertices of $\mathcal{T}$, and associating to any $g \in \mathcal{G}_4$ the permutation induced on the vertices themselves. In order to define the function $\varphi$ satisfying ($i$)-($iii$) of Definition \ref{def: admissible pair}, we first take a tetrahedron with barycenter in $0$, and define on a face $A$ a positive function $\tilde \varphi_1$ being $0$ on $\pa A$, and being symmetric with respect to all the transformations in $\mathcal{G}_4$ leaving invariant $A$. By rotation, we can define $\tilde \varphi_2$, $\tilde \varphi_3$ and $\tilde \varphi_4$ on the remaining faces. Now, considering the radial projection of the tetrahedron into the unit sphere $\S^{2}$, we obtain a function $(\varphi_1,\dots,\varphi_4)$ whose $1$-homogeneous extension is by construction $(\mathcal{G}_4,h_4)$-equivariant, and satisfies ($i$)-($iii$) of Definition \ref{def: admissible pair}. Thus $(4,\mathcal{G}_4,h_4)$ is an admissible triplet, and Theorem \ref{thm: new existence} yields the existence of a $(\mathcal{G}_4,h_4)$-equivariant solution for the system  with $4$ components in $\R^3$. Since the symmetries of the tetrahedron involve the dependence on $3$ variables, this solution is not $2$-dimensional. 

In a similar way, one can construct $(\mathcal{G}_6,h_6)$-equivariant solutions with respect to the group of symmetries of the cube $\mathcal{G}_6$ (isomorphic to a subgroup of $\mathfrak{S}_8$ through a isomorphism $h_6$) for systems with $k=3$ or $k=6$ components. To this purpose, we consider a cube with barycenter in $0$ in $\R^3$, and we define on a face a positive function $\tilde \varphi_1$ being $0$ on $\pa A$, and being symmetric with respect to all the transformations in $\mathcal{G}_6$ leaving invariant $A$. By rotation, we can define $\tilde \varphi_2, \dots, \tilde \varphi_6$ on the remaining faces. Considering the radial projection of the cube into the unit sphere $\S^{2}$, we obtain a function $(\varphi_1,\dots,\varphi_6)$ whose $1$-homogeneous extension is $(\mathcal{G}_6,h_6)$-equivariant and satisfies ($i$)-($iii$) of Definition \ref{def: admissible pair}. Theorem \ref{thm: new existence} gives then a $3$-dimensional $(\mathcal{G}_6,h_6)$-equivariant solution to \eqref{entire system} with $6$ components in $\R^3$. In order to obtain a $3$-components $(\mathcal{G}_6,h_6)$-equivariant solution, we proceed as in the previous discussion replacing $\tilde \varphi_1$ with $\tilde \psi_1 = \tilde \varphi_1 + \tilde \varphi_4$, where $\varphi_4$ has support on the face opposite to $A$ in the cube. By rotation, we determine $\tilde \psi_2$ and $\tilde \psi_3$, each of them supported on the union of two opposite faces. As before, we can then consider the radial projection onto $\S^2$, and afterwards its $1$-homogeneous extension $(\psi_1,\psi_2,\psi_3)$, which is $(\mathcal{G}_6,h_6)$-equivariant and satisfies ($i$)-($iii$) of Definition \ref{def: admissible pair}. For the equivariance, we recall that any isometry of the cube is identified by the faces three given adjacent faces are mapped to (this is why we could construct solutions with cubical symmetry for systems with $3$ components). In conclusion, by Theorem \ref{thm: new existence} we obtain a $(\mathcal{G}_6,h_6)$-equivariant solution of \eqref{entire system} with $k=3$ components. 

Arguing in a similar way, we may also obtain equivariant solutions with respect to the symmetries of the octahedron for systems with $k=4$ and $k=8$ components, and so on.

\section{An Alt-Caffarelli-Friedman monotonicity formula for equivariant solutions}\label{sec: acf}

In the rest of the section we aim at proving Proposition \ref{prop: acf equiv}. We always suppose that $(k,\mathcal{G},h)$ is an admissible triplet, according to Definition \ref{def: admissible pair}. Moreover, we often omit the mention ``up to a subsequence" for simplicity.
%For $\mf{U} \in H^1(\R^N,\R^k)$ and $i=1,\dots,k$ we recall definition \eqref{acf functional factor}:
%\[
%J_i(r):= \int_{B_r} \frac{|\nabla U_i|^2 + U_i^2 \sum_{j \neq i} U_j^2}{|x|^{N-2}}.
%\]
%
%We present a $k$-phases monotonicity formula for $(\mathcal{G},h)$-equivariant function; we recall that $\ell(k,\mathcal{G},h)$ has been defined in \eqref{def equivariant optimal value}.
%\begin{proposition}\label{prop: acf equiv}
%There exists a constant $C>0$ depending only on $N$ and on $(k,\mathcal{G},h)$ such that, for any  $(\mathcal{G},h)$-equivariant solution $\mf{U}$ of \eqref{entire system}, the function
%\[
%r \mapsto \frac{1}{r^{2k \ell(k,\mathcal{G},h)}} e^{-Cr^{-1/2}} J_1(r) \cdots J_k(r)
%\]
%is monotone non-decreasing for $r > 1$. 
%\end{proposition}
The proof is divided in several steps, and, as usual when dealing with Alt-Caffarelli-Friedman monotonicity formulae for competing systems, is based upon a control on an ``approximated" optimal partition problem on $\S^{N-1}$. For any $\mf{u} \in H^1(\S^{N-1},\R^k)$, we let
\[
	I_\beta(\mbfu) := \frac{1}{k} \sum_{i=1}^{k}\gamma \left( \frac{\int_{\S^{n-1}} |\nabla_\theta u_i|^2 + \frac12 \beta u_i^2 \sum_{j \neq i} u_j^2}{\int_{\S^{n-1}} u_i^2}\right),
\]
where 
\[
\gamma(t):= \sqrt{\left(\frac{N-2}{2}\right)^2 + t} - \left(\frac{N-2}{2}\right).
\]
We denote by $\hat H_{(\mathcal{G},h)}$ the subspace of $(\mathcal{G},h)$-equivariant functions in $H^1(\S^{N-1},\R^k)$, and we introduce the optimal value 
\[
\ell_\beta(k,\mathcal{G},h):= \inf_{\hat H_{(\mathcal{G},h)}} I_\beta.
\]
In what follows, to keep the notation as simple as possible, we simply write $\ell$ and $\ell_\beta$ instead of $\ell(k,\mathcal{G},h)$ and $\ell_\beta(k,\mathcal{G},h)$, respectively.

%If needed, we shall refer to $j_\beta$ as the \emph{free} functional, while to $j_{\beta}|_{\mathcal{H}}$ (that is, the functional evaluated on the set $\mathcal{H}_{(H,G)}(\S^{n-1})$ as the \emph{restricted} functional.

\begin{lemma}\label{lem: problemi di minimo}
Both $\ell$ and $\ell_\beta$ are positive and achieved (for all $\beta>0$). It results $\ell_\beta \to \ell$ as $\beta \to +\infty$, and there exists a minimizer for $\ell_\beta$ which solves
\begin{equation}\label{eqn: EL}
\begin{cases}
-\Delta_{\theta} u_{i,\beta} = \lambda_\beta u_{i,\beta} - \beta u_{i,\beta} \sum_{j \neq i} u_j^2 & \text{in $\S^{N-1}$} \\
u_{i,\beta}>0 & \text{in $\S^{N-1}$} \\
\int_{\S^{N-1}} u_{i,\beta}^2=1 & \qquad \forall i,
\end{cases}
\end{equation}
where $\lambda_\beta \geq 0$, and $\Delta_\theta$ denotes the Laplace-Beltrami operator on $\S^{N-1}$. Moreover, $\mf{u}_\beta \wc \varphi$ weakly in $H^1(\S^{N-1},\R^k)$, and $\varphi$ is a nonnegative minimizer for $\ell$.
\end{lemma}

%For any $(H,G)$ falling under assumptiomn (A) and $\beta >0$, there exists a minimizer $\mbfu_\beta \in \mathcal{H}_{(H,G)}(\S^{n-1})$ of the restricted  $j_\beta|_{\mathcal{H}}$
%\[
%	d_\beta := j_{\beta}(\mbfu_\beta) = \inf_{\mbfu \in \mathcal{H}_{(H,G)}(\S^{n-1})} j_{\beta}(\mbfu).
%\]
%Moreover $\mbfu_\beta$ is a critical point of $j_\beta$ also in the space $H^1(\S^{n-1};\R^k)$ and it satisfies
%\begin{equation}\label{eqn: EL}
%	-\Delta_{LB} u_{i,\beta} = \lambda_\beta u_{i,\beta} - \beta u_{i,\beta} \sum_{j \neq i} u_j^2
%\end{equation}
%where $\Delta_{LB}$ is the Laplace-Beltrami operator on $\S^{n-1}$ and $\lambda_\beta > 0$ are non-negative parameters, uniformly bounded in $\beta$.
%\end{lemma}
\begin{proof}
Restricting ourselves to the subset of functions in $\hat H_{(\mathcal{G},h)}$ whose components have prescribed $L^2(\S^{N-1})$-norm equal to $1$, it is easy to check that the functional $I_\beta$ is weakly lower semi-continuous and coercive. Since $\hat H_{(\mathcal{G},h)}$ is also weakly closed, the direct method of the calculus of variations ensures the existence of a minimizer $\mf{u}_\beta$ for $\ell_\beta$, which can be assumed to be nonnegative. By the Palais' principle of symmetric criticality (notice that $I_\beta$ is invariant under the action of any symmetry in $\mathcal{O}(N)$), the Lagrange multipliers rule, and the strong maximum principle, it follows that $\mf{u}_\beta$ satisfies
\[
\begin{cases}
-\Delta_{\theta} u_{i,\beta} + \sum_{j \neq i}\frac12 \left( 1+ \frac{\mu_{j,\beta}}{\mu_{i,\beta}} \right) \beta u_{i,\beta} u_{j,\beta}^2 = \lambda_{i,\beta} u_{i,\beta} & \text{in $\S^{N-1}$} \\
u_{i,\beta}>0 & \text{in $\S^{N-1}$},
\end{cases}
\]
where
\[
\mu_{i,\beta}:= \gamma'\left( \int_{\S^{n-1}} |\nabla_\theta u_{i,\beta}|^2 + \frac12 \beta u_{i,\beta}^2 \sum_{j \neq i} u_{j,\beta}^2\right).
\]
The equation for $u_{i,\beta}$ is nothing but \eqref{eqn: EL}: indeed, thanks to the symmetries in $\hat H(\mathcal{G},h)$ (see Remark \ref{rem: ne basta una}), we have $\mu_{i,\beta}=\mu_{j,\beta}$ and $\lambda_{i,\beta}=\lambda_{j,\beta} \geq 0$ for every $i \neq j$. Finally, $\ell_\beta>0$ since otherwise $\mf{u}_\beta \equiv \mf{0}$, in contradiction with the normalization condition.

As far as $\ell$ is concerned, we introduce an auxiliary functional $I_\infty : \hat H_{(\mathcal{G},h)} \to (0,+\infty]$ defined by
\[
	I_\infty(\mbfu) := 
	\begin{cases}
		\frac{1}{k} \sum_{i=1}^{k}\gamma \left( \frac{\int_{\S^{n-1}} |\nabla u_i|^2 }{\int_{\S^{n-1}} u_i^2}\right) &\text{if $u_iu_j = 0$ a.e. on $\S^{n-1}$ for any $i \neq j$}\\
		+\infty &\text{otherwise}.
	\end{cases}
\]
It is easy to see that $I_\beta$ is increasing in $\beta$ and converges pointwise to $I_\infty$, implying that $I_\infty$ is a weakly lower semi-continuous functional in the weakly closed set $\hat H_{(\mathcal{G},h)}$, and that $I_\beta$ $\Gamma$-converges to $I_\infty$ in the weak $H^1$-topology. Moreover, being the family $\{I_\beta\}$ equi-coercive, any sequence $\{\mbfu_\beta\}$ of minimizers for $I_\beta$ converges to a minimizer $\mbfu$ of $I_\infty$. Finally, by definition, $\ell>\ell_\beta$ for every $\beta>0$, whence $\ell>0$ follows.
\end{proof}

Further properties of the sequence $\{\mf{u}_\beta\}$ are collected in the next two lemmas.

\begin{lemma}\label{lem: boundedness}
The sequence $\{\mf{u}_\beta\}$ is uniformly bounded in $\Lip(\S^{N-1})$. Moreover, the sequence $(\lambda_\beta)$ is bounded.
\end{lemma}
\begin{proof}
Let $\{\mf{u}_\beta\}$ be a sequence of minimizers for $\ell_\beta$ satisfying \eqref{eqn: EL}, weakly converging to a minimizer $\mf{u}$ for $\ell$. As $I_\beta(\mf{u}_\beta) = \ell_\beta \le \ell$, there exists $C>0$ such that
\[
\int_{\S^{N-1}} \beta u_{i,\beta}^2 \sum_{j \neq i}  u_{j,\beta}^2 \le C.
\]
Moreover, by weak convergence, $\{\mf{u}_\beta\}$ is bounded in $H^1(\S^{N-1},\R^k)$. Therefore, testing the first equation in \eqref{eqn: EL} against $u_{i,\beta}$, we deduce that $\{\lambda_\beta\}$ is a bounded sequence of positive numbers, and this implies, through a Brezis-Kato argument (see for instance \cite[Page 124]{Tavares} for a detailed proof and \cite{BrKa} for the original argument), that $\{\mf{u}_\beta\}$ is uniformly bounded in $L^\infty(\S^{N-1},\R^k)$. By the main results in \cite{SoZi}, we infer that $\{\mf{u}_\beta\}$ is uniformly bounded\footnote{
%\textcolor{red}{It is worth mentioning that the results in \cite{SoZi} are proved for the Laplace operator in the interior of subsets of  $\R^N$, and their extension to a Riemaniann setting is not straightforward; the general extension of \cite{SoZi} to equations on manifolds will be the object a future contribution \cite{SoZiFuturo}. We anticipate here the main argument: the key ingredients for the results in \cite{SoZi} are elliptic estimates, an Almgren-type monotonicity formula, and an Alt-Caffarelli-Friedman-type monotonicity formula. Thus, to extend \cite{SoZi} on $\S^{N-1}$ we need to extend these three tools for equations on $\S^{N-1}$. The elliptic theory is already available. The Almgren-type monotonicity formula can be obtained by combining the ones in \cite{GarLin} (Almgren-type monotonicity formula for scalar equations with variable coefficients) and in \cite{SoZi} (Almgren-type monotonicity formula for systems driven by the Laplacian). A similar combination was already done in a slightly different context in \cite[Section 7]{tt}. The Alt-Caffarelli-Friedman-type monotonicity formula can be obtained by combining the ones in \cite{TeixZha} (Alt-Caffarelli-Friedman-type monotonicity formula for scalar equations on Riemaniann manifold) and in \cite{SoZi} (Alt-Caffarelli-Friedman-type monotonicity formula for systems driven by the Laplacian). Once that these three tools are available, the proof proceeds as in \cite{SoZi}.}
It is worth mentioning that the results in \cite{SoZi} are proved for the Laplace operator in the interior of subsets of  $\R^N$, and their extension to a Riemaniann setting presents some technical difficulties; the general extension of \cite{SoZi} to equations on manifolds will be the object a future contribution \cite{SoZiFuturo}. We anticipate here the main argument: the key ingredients for the regularity results in \cite{SoZi} are elliptic estimates, an Almgren-type monotonicity formula and a sharp version of the Alt-Caffarelli-Friedman-type monotonicity formula. Thus, we need to extend these three tools for systems on $\S^{N-1}$. The elliptic theory is already available, as the Almgren-type monotonicity formula (see for instance \cite[Section 7]{tt}). The Alt-Caffarelli-Friedman-type monotonicity formula represents the only obstruction, but it can be obtained by combining the results in \cite{TeixZha} (Alt-Caffarelli-Friedman-type monotonicity formula for scalar equations on Riemaniann manifold) and in \cite{SoZi} (the sharp version of Alt-Caffarelli-Friedman-type monotonicity formula for systems in the euclidean space). Once that these three tools are available, the proof proceeds as in \cite{SoZi}.
} in $\Lip(\S^{N-1})$.
\end{proof}

\begin{lemma}\label{lem: problem solved by minimizer}
We have $\mf{u}_\beta \to \varphi$ strongly in $H^1(\S^{N-1})$ topology, in $\mathcal{C}^{0,\alpha}(\S^{N-1})$ for every $0<\alpha<1$, and
\[
\lim_{\beta \to +\infty} \beta \int_{\S^{N-1}} u_{i,\beta}^2 u_{j,\beta}^2 = 0.
\]
Moreover $\lambda_\beta \to \ell(\ell + N-2)$, and 
\[
\begin{cases}
-\Delta_\theta \varphi_i = \ell(\ell+N-2) \varphi_i & \text{in $\{\varphi_i >0\}$} \\
\int_{\S^{N-1}} \varphi_i^2 = 1.
\end{cases}
\]
\end{lemma}
\begin{proof}
Thanks to Lemma \ref{lem: boundedness}, we can simply apply Theorem 1.4 in \cite{NoTaTeVe}. To check that $\lambda_\beta \to \ell(\ell+N-2)$, we observe that by boundedness $\lambda_\beta \to \lambda_\infty \ge 0$ as $\beta \to +\infty$. Therefore, recalling that $\mf{u}_\beta \wc \varphi$ in $H^1(\S^{N-1},\R^k)$, for $i=1,\dots,k$ we have
\[
\begin{cases}
-\Delta_\theta \varphi_i = \lambda_\infty \varphi_i \qquad \text{in $\{\varphi_i >0\}$} \\
\int_{\S^{N-1}} \varphi_i^2 = 1.
\end{cases}
\]
This implies that
\begin{align*}
\ell &= \frac{1}{k}\sum_i \sqrt{\left( \frac{N-2}{2}\right)^2+\int_{\S^{N-1}} |\nabla_\theta \varphi_i|^2} - \frac{N-2}{2} \\
&= \sqrt{\left( \frac{N-2}{2}\right)^2+\lambda_\infty} - \frac{N-2}{2},
\end{align*}
whence the thesis follows.
\end{proof}

The following result is the counterpart of Lemma 4.2 in \cite{Wa} in a $(\mathcal{G},h)$-equivariant setting, see also Theorem 5.6 in \cite{BeTeWaWe} for an analogue statement in dimension $N=2$.

\begin{lemma}\label{lem: opt approximated}
There exists a constant $C > 0$ such that
\[
	\ell_\beta \ge  \ell-  C \beta^{-1/4}.
\]
\end{lemma}

Before proving the lemma, we need a technical result. We recall that $\hat H_{(\mathcal{G},h)}$ denotes the set of $(\mathcal{G},h)$-equivariant functions in $H^1(\S^{N-1},\R^k)$. 

\begin{lemma}\label{lem: cappucci equivariant}
Let $\mf{u} \in \hat H_{(\mathcal{G},h)}$. Then also the function $\hat{\mf{u}}$, defined by 
\[
\hat u_{i} = v_{i}^+:= \left( u_{i} - \sum_{j \neq i} u_{j,}\right)^+,
\]
belongs to $\hat H_{(\mathcal{G},h)}$.
\end{lemma}
\begin{proof}
As $u_{i} \in H^1(\S^{N-1})$, it follows straightforwardly that $\hat{\mf{u}} \in H^1(\S^{N-1},\R^k)$. We have to show that it is also $(\mathcal{G},h)$-equivariant, and to this aim it is sufficient to show that $\mf{v}$ is $(\mathcal{G},h)$-equivariant. This can be checked directly:
\begin{align*}
v_{(h(g))^{-1}(i)}(g(x)) & = u_{(h(g))^{-1}(i)}(g(x)) - \sum_{j \neq (h(g))^{-1}(i)} u_j(g(x)) \\
& =  u_{(h(g))^{-1}(i)}(g(x)) - \sum_{j \neq i} u_{(h(g))^{-1}(j)}(g(x)) = v_{i}(x),
\end{align*}
where the last equality follows by the fact that $\mf{u}$ is $(\mathcal{G},h)$-equivariant.
\end{proof}

\begin{proof}[Proof of Lemma \ref{lem: opt approximated}]
In order to simplify the notation, only in this proof we write $\nabla$ and $\Delta$ instead of $\nabla_\theta$ and $\Delta_\theta$, respectively.
Let us consider the functions $\hat{\mf{u}}_\beta$, defined in Lemma \ref{lem: cappucci equivariant}. Since the components of $\hat{\mf{u}}_\beta$ have disjoint supports, we can use it as competitor for $\ell$. We aim at showing that $\hat{\mf{u}}_\beta$ is actually close enough to $\mf{u}_\beta$ in the energy sense, and in doing this we shall use many times the properties proved in Lemma \ref{lem: boundedness}. To be precise, we shall prove that there exists a constant $C>0$ such that
\begin{align}
	1 - C \beta^{-1/2} \leq \int_{\S^{n-1}} \hat u_{i,\beta}^2 \leq 1 + C \beta^{-1/2}  \label{vicinanza L2},\\
	\int_{\S^{N-1}} |\nabla \hat u_{i,\beta} |^2 \leq \int_{ \S^{N-1} } |\nabla u_{i,\beta}|^2 + C \beta^{-1/4}. \label{vicinanza H1} 
\end{align}
Before we continue, let us point out that second estimate can be derived from an analogous one, stated as follow: there exists $C>0$ independent of $\beta$ and $\bar \delta>0$ such that for almost any $\delta \in (0, \bar \delta)$ we have
\[
	\int_{\{\hat u_{i,\beta} > \delta\} } |\nabla \hat u_{i,\beta} |^2 \leq \int_{ \S^{N-1} } |\nabla u_{i,\beta}|^2 + C \beta^{-1/4} + C \delta. 
\]
Indeed, if the previous estimate is satisfied,
\begin{align*}
	\int_{\S^{N-1}} |\nabla \hat u_{i,\beta} |^2 & = \int_{\{\hat u_{i,\beta} > 0\} } |\nabla \hat u_{i,\beta} |^2 = \lim_{\delta \to 0^+} \int_{\{\hat u_{i,\beta} > \delta\} } |\nabla \hat u_{i,\beta} |^2 \\
	&  \leq \int_{ \S^{N-1} } |\nabla u_{i,\beta}|^2 + C \beta^{-1/4}.
\end{align*}
Notice that in principle the value $\bar \delta$ could depend on $\beta$, but this is not a problem since $C$ is, on the contrary, a universal constant.

\noindent \textbf{Pointwise bounds.}
The boundedness of $\{\mf{u}_\beta\}$ in $\Lip(\S^{N-1})$, Lemma \ref{lem: boundedness}, implies that there exists a constant $C_1 > 0$ such that
\begin{equation}\label{upper estimate on sphere}
	\beta^{1/2} u_{i,\beta} u_{j,\beta} \leq C_1 \qquad \forall i \neq j.
\end{equation}
The proof is a straightforward adaptation of the one in \cite[Theorem 1.1]{SoZi2}, which regards the same estimate in subsets of $\R^N$.

As a consequence we have that for each $\theta \in \S^{N-1}$ and each $\beta > 0$ 
\begin{equation}\label{lower estimate just one}
\text{there exists at most one index $i$ such that $u_{i,\beta}(\theta) \geq 2 k C_1^{1/2} \beta^{-1/4}$},
\end{equation}
where $C_1$ is the same constant appearing in \eqref{upper estimate on sphere}. Indeed, assuming the contrary, there would exists two distinct indices $i \neq j$ satisfying the previous inequality, and hence
\[
	4 k^2 C_1 \beta^{-1/2} \leq u_{i,\beta}(\theta) u_{j,\beta}(\theta) \leq C_1 \beta^{-1/2},
\]
a contradiction.

Finally, we observe that
\begin{equation}\label{la seconda non batte la somma}
\text{if $\hat u_{i,\beta}(\theta) = 0$, then $u_{i,\beta}(\theta) \le 2 k (k-1) C_1^{1/2} \beta^{-1/4}$.}
\end{equation}
If not, we have that \eqref{lower estimate just one} holds for $i$, and moreover
\[
 2 k (k-1) C_1^{1/2} \beta^{-1/4} \le u_{i,\beta} (\theta) \le \sum_{j \neq i} u_{j,\beta}(\theta) \le (k-1) \max_{j \neq i} u_{j,\beta}(\theta);
\]
hence there exist two indexes for which \eqref{lower estimate just one} is satisfied in $\theta$, a contradiction.

\noindent \textbf{Integrals bounds for the Laplacian.}
We prove that there exists a constant $C > 0$ (independent of $\beta$) such that
\begin{equation}\label{bounds on lap}
	\int_{\S^{N-1}} |\Delta u_{i,\beta} | \leq C.
\end{equation}
Indeed, directly form the equation and the divergence theorem
\[
	0 = \int_{\S^{N-1}} (-\Delta u_{i,\beta} )= \int_{\S^{N-1}} \lambda_{\beta} u_{i,\beta} - \beta u_{i,\beta} \sum_{j \neq i} u_j^2;
\]
that is
\[
	0 \leq \int_{\S^{N-1}} \beta u_{i,\beta} \sum_{j \neq i} u_{j,\beta}^2 =  \int_{\S^{N-1}} \lambda_{\beta} u_{i,\beta} \leq C,
\]
as the functions $u_{i,\beta}$ are boudned in $L^\infty(\S^{N-1})$, and $\{ \lambda_{\beta} \}$ is  bounded. Consequently
\[
	\int_{\S^{N-1}} |\Delta u_{i,\beta}| \leq \int_{\S^{N-1}} \lambda_{\beta} u_{i,\beta} + \beta u_{i,\beta} \sum_{j \neq i} u_{j,\beta}^2 \leq C.
\]

\noindent \textbf{Integrals bounds for the competition term.}
Using \eqref{lower estimate just one} and the computations in the previous point, we deduce that
\begin{align*}
	\int_{\S^{N-1}} \beta \sum_{i \neq j} u_{i,\beta}^2 u_{j,\beta}^2 &\le \sum_{i \neq j} \left( \|u_{i,\beta}\|_{L^\infty(\{ u_{i,\beta} \le u_{j,\beta}\})}   \int_{ \{u_{i,\beta} \le u_{j,\beta}\}} \beta u_{i,\beta} u_{j,\beta}^2 \right. \\
	 &\left.+  \|u_{j,\beta}\|_{L^\infty(\{ u_{j,\beta} < u_{i,\beta}\})}   \int_{ \{u_{j,\beta} < u_{i,\beta}\}} \beta u_{j,\beta} u_{i,\beta}^2 \right)   \\
	   &\le C \beta^{-1/4} \sum_{i=1}^k \int_{ \{u_{i,\beta} \le u_{j,\beta}\}} \beta u_{i,\beta} \sum_{j \neq i}u_{j,\beta}^2 \leq C \beta^{-1/4}.
\end{align*}

\noindent \textbf{Integrals bounds for the normal derivatives.}
For analogous reasons, we can show that there exists a constant $C>0$ and $\bar \delta>0$ small enough such that for almost every $\delta \in (0,\bar \delta )$ it holds
\[
	\int_{\partial \{\hat u_{i,\beta} > \delta\}} | \partial_{\nu} \hat u_{i,\beta} |  \leq C.
\]
Firstly, since for $\beta$ fixed the function $\hat u_{i,\beta}$ is regular, the set $\partial \{\hat u_{i,\beta} > \delta\}$ is regular for almost every $\delta>0$, by Sard's Lemma. Moreover, since $\hat u_{i,\beta}$ is nonnegative and regular, if $\delta< \bar \delta$ is small enough
\begin{equation}\label{eq 7000}
	\int_{\partial \{\hat u_{i,\beta} > \delta \}} | \partial_{\nu} \hat u_{i,\beta} |  = -\int_{\partial \{\hat u_{i,\beta} > \delta \}}  \partial_{\nu} \hat u_{i,\beta}.
\end{equation}
Hence for almost every $\delta \in (0,\bar \delta)$ the set $\partial \{\hat u_{i,\beta} > \delta\}$ is regular, and \eqref{eq 7000} holds. With this choice we are in position to apply the divergence theorem, and consequently
\[
	\left| \int_{\partial \{\hat u_{i,\beta} > \delta \}}  \partial_{\nu} \hat u_{i,\beta} \right| = \left| \int_{ \{\hat u_{i,\beta} > \delta \}} \Delta \hat u_{i,\beta} \right| \leq  \int_{ \{\hat u_{i,\beta} > \delta \}} \sum_{j=1}^k |\Delta u_{j,\beta}| \leq C,
\]
where $C$ is independent of $\beta$ by \eqref{bounds on lap}. With similar computations we have also the uniform estimate
\[
	\left|\int_{\partial \{\hat u_{i,\beta} > \delta\}} \partial_{\nu}  u_{i,\beta}\right| \leq C.
\]

\noindent \textbf{Estimates for the $L^2(\S^{N-1})$ norm.}
Thanks to \eqref{lower estimate just one} and \eqref{la seconda non batte la somma}, we have
\begin{multline*}
		\int_{\S^{n-1}} (\hat u_{i,\beta} - u_{i,\beta})^2 = \int_{\{\hat u_{i,\beta} > 0\}} (\hat u_{i,\beta} - u_{i,\beta})^2 +  \int_{\{\hat u_{i,\beta} = 0\}} (\hat u_{i,\beta} - u_{i,\beta})^2\\
%		&= \int_{ \{u_{i,\beta} > \sum_{j \neq i} u_{j,\beta}\} } (\hat u_{i,\beta} - u_{i,\beta})^2 +  %\int_{\{u_{i,\beta} \leq \sum_{j \neq i} u_{j,\beta}\}} u_{i,\beta}^2\\
		= \int_{ \{u_{i,\beta} > \sum_{j \neq i} u_{j,\beta}\}} \left(\sum_{j \neq i} u_{j,\beta}\right)^2 +  \int_{\{\hat u_{i,\beta} = 0\}} u_{i,\beta}^2 \leq C \beta^{-1/2},
\end{multline*}
whence \eqref{vicinanza L2} follows.

\noindent \textbf{Estimates for the $H^1(\S^{N-1})$ seminorm.}
As a last step, we wish to estimate the $L^2$ norm of $ \nabla \hat u_{i,\beta}$. Since $\pa \{\hat u_{i,\beta}>\delta\}$ is regular, we can apply the divergence theorem deducing that
\begin{align*}
	\int_{ \{\hat u_{i,\beta} > \delta \}} |\nabla \hat u_{i,\beta} |^2  &= \int_{ \{\hat u_{i,\beta} > \delta\} } (- \Delta \hat u_{i,\beta}) \hat u_{i,\beta} + \int_{ \partial\{\hat u_{i,\beta} > \delta\} } (\partial_\nu \hat u_{i,\beta}) \hat u_{i,\beta} \\
%	 =  \int_{ \{\hat u_{i,\beta} > \delta\} } (- \Delta u_{i,\beta}) \hat u_{i,\beta} +  \int_{\{ \hat u_{i,\beta} > \delta\} } \Delta \left(\sum_{j \neq i} u_{j,\beta}\right) \hat u_{i,\beta} + \delta \int_{ \partial\{\hat u_{i,\beta} > \delta\} } \partial_\nu \hat u_{i,\beta}\\
	 &= \underbrace{ \int_{\{\hat u_{i,\beta} > \delta\} } (- \Delta u_{i,\beta}) u_{i,\beta}}_{\text{(I)}}  +  \int_{ \{\hat u_{i,\beta} > \delta\} }\Delta u_{i,\beta} \sum_{j \neq i} u_{j,\beta}  \\ &+ \underbrace{\int_{ \{\hat u_{i,\beta} > \delta\} } \Delta\left(\sum_{j \neq i} u_{j,\beta}\right) \hat u_{i,\beta}}_{\text{(II)}} + \delta \int_{ \partial\{\hat u_{i,\beta} > \delta\} } \partial_\nu \hat u_{i,\beta}
\end{align*}
The first term (I) can be bounded using the equation, also recalling that $\lambda_\beta\geq0$:
\[
\begin{split}
	\int_{ \{\hat u_{i,\beta} > \delta\} } &  (- \Delta u_{i,\beta}) u_{i,\beta} = \int_{ \{\hat u_{i,\beta} > \delta\} } \lambda_{\beta} u_{i,\beta}^2 - \beta u_{i,\beta}^2 \sum_{j \neq i} u_{j,\beta}^2 \\
	\le &\int_{ \S^{N-1} } \lambda_{\beta} u_{i,\beta}^2 - \beta u_{i,\beta}^2 \sum_{j \neq i} u_{j,\beta}^2 +  \int_{ \S^{N-1} \setminus  \{\hat u_{i,\beta} > \delta \} } \beta u_{i,\beta}^2 \sum_{j \neq i} u_{j,\beta}^2  \\
	 = &\int_{ \S^{N-1} } |\nabla u_{i,\beta}|^2 +  \int_{ \S^{N-1} \setminus \{\hat u_{i,\beta} > \delta \} } \beta u_{i,\beta}^2 \sum_{j \neq i} u_{j,\beta}^2.
\end{split}
\]
The term (II) can be expanded further as
\begin{multline*}
	\int_{ \{\hat u_{i,\beta} > \delta\} } \Delta\left(\sum_{j \neq i} u_{j,\beta}\right) \hat u_{i,\beta} = -\int_{ \{\hat u_{i,\beta} > \delta\} }  \nabla \left(\sum_{j \neq i} u_{j,\beta}\right) \cdot \nabla\hat u_{i,\beta} \\
 + \delta \int_{ \partial\{\hat u_{i,\beta} > \delta\} } \partial_\nu \left(\sum_{j \neq i} u_{j,\beta}\right) 
 = \int_{ \{\hat u_{i,\beta} > \delta\} } \left(\sum_{j \neq i} u_{j,\beta}\right) \Delta \hat u_{i,\beta} \\
 - \int_{ \partial \{ \hat u_{i,\beta} > \delta \} } \left(\sum_{j \neq i} u_{j,\beta}\right)  \partial_{\nu}\hat u_{i,\beta} + \delta \int_{ \partial\{\hat u_{i,\beta} > \delta\} } \partial_\nu \left(\sum_{j \neq i} u_{j,\beta}\right).
\end{multline*}
Recollecting the previous computations, and using again \eqref{lower estimate just one}, we have
\begin{multline*}
	\int_{ \{\hat u_{i,\beta} > \delta\} } |\nabla \hat u_{i,\beta} |^2 \leq \int_{ \S^{N-1} } |\nabla u_{i,\beta}|^2 +  \int_{ \S^{N-1} \setminus \{ \hat u_{i,\beta} > \delta \} } \beta u_{i,\beta}^2 \sum_{j \neq i} u_{j,\beta}^2  \\ 
	+  \int_{ \{\hat u_{i,\beta} > \delta\} }\Delta u_{i,\beta} \sum_{j \neq i} u_{j,\beta} + \int_{ \{\hat u_{i,\beta} > \delta\} } \left(\sum_{j \neq i} u_{j,\beta}\right) \Delta \hat u_{i,\beta} \\
	- \int_{ \partial \{ \hat u_{i,\beta} > \delta \} } \left(\sum_{j \neq i} u_{j,\beta}\right)  \partial_{\nu}\hat u_{i,\beta}
	+ \delta \int_{ \partial\{\hat u_{i,\beta} > \delta\} } \partial_\nu u_{i,\beta} \\
	% \leq \int_{ \S^{N-1} } |\nabla u_{i,\beta}|^2 + C \beta^{-1/4} \left( 1 +  \int_{ \hat u_{i,\beta} > \delta }|\Delta u_{i,\beta}| + \int_{ \hat u_{i,\beta} > \delta } |\Delta \hat u_{i,\beta}| + \int_{ \partial \{ \hat u_{i,\beta} > \delta \} } |\partial_{\nu}\hat u_{i,\beta}| \right) \\
	% + C \delta\\
	\leq \int_{ \S^{N-1} } |\nabla u_{i,\beta}|^2 + C \beta^{-1/4} + C \delta,
\end{multline*}
which, as already observed, implies \eqref{vicinanza H1}.

With \eqref{vicinanza L2} and \eqref{vicinanza H1} we are in position to complete the proof. By minimality $\ell \le I_\infty(\hat{\mf{u}}_\beta)$ for every $\beta$, which gives
\begin{align*}
	\ell &\leq \frac{1}{k} \sum_{i = 1}^{k} \gamma\left(\frac{\int_{ \S^{N-1} } |\nabla \hat u_{i,\beta}|^2}{\int_{ \S^{N-1} } \hat u_{i,\beta}^2}\right) \leq  \frac{1}{k} \sum_{i = 1}^{k} \gamma\left(\frac{\int_{ \S^{N-1} } |\nabla u_{i,\beta}|^2 + C\beta^{-1/4} }{1- C\beta^{-1/2}}\right) \\
	&\leq \frac{1}{k} \sum_{i = 1}^{k} \gamma\left(\int_{ \S^{N-1} } |\nabla u_{i,\beta}|^2 + \frac12 \beta u_{i,\beta}^2 \sum_{j \neq i} u_{j, \beta}^2\right) +  C\beta^{-1/4} = \ell_{\beta} +  C\beta^{-1/4}. \qedhere
\end{align*}
\end{proof}

The proof of Proposition \ref{prop: acf equiv} can be obtained in a somehow usual way. 

\begin{proof}[Sketch of the proof of Proposition \ref{prop: acf equiv}]
Arguing as in Section 7 of \cite{ctv}, or \cite[Lemma 2.5]{NoTaTeVe}, or else \cite[Theorem 3.14]{SoZi}, it is possible to check that 
\[
\frac{d}{dr}\log \left( \frac{J_1(r) \cdots J_k(r)}{r^{2k\ell}}\right) = -\frac{2k\ell}{r} + \frac{2}{r}\sum_{i} \gamma \left( \frac{r^2\int_{\pa B_r} |\nabla u_i|^2 +  \frac12 u_i^2 \sum_{j \neq i} u_j^2}{\int_{\pa B_r} u_i^2}\right).
\]
Changing variables in the integrals (see Theorem 3.14 in \cite{SoZi} for the details), we deduce that 
\[
\sum_{i} \gamma \left( \frac{r^2\int_{\pa B_r} |\nabla u_i|^2 + \frac12 u_i^2 \sum_{j \neq i} u_j^2}{\int_{\pa B_r} u_i^2}\right) \ge k \ell_{r^2},
\]
where $\ell_{r^2}$ denotes the optimal value $\ell_\beta$ for $\beta=r^2$. Coming back to the previous equation, and using Lemma \ref{lem: opt approximated}, we conclude that
\[
\frac{d}{dr}\log \left( \frac{J_1(r) \cdots J_k(r)}{r^{2k\ell}}\right)  \ge \frac{2k}{r}(\ell_{r^2}-\ell)  \ge -2kC r^{-3/2},
\]
and integrating the thesis follows.
\end{proof}

\section{Construction of equivariant solutions}\label{sec: main}

For an admissible triplet $(k,\mathcal{G},h)$, we prove the existence of a $(\mathcal{G},h)$-equivariant solution to \eqref{entire system} with $k$ components. We partially follow the method introduced in \cite{BeTeWaWe}, which consists in two steps:
\begin{itemize}
\item firstly, we prove the existence of a sequence of $(\mathcal{G},h)$-equivariant solutions $\mf{V}_R$, defined in balls of increasing radii $R \to +\infty$;
\item secondly, we show that such sequence converges locally uniformly in $\R^N$ to a nontrivial solution.
\end{itemize}
With respect to \cite{BeTeWaWe}, we modify the construction conveniently choosing $R$ from the beginning; this simplifies substantially the proof of the convergence of $\{\mf{V}_R\}$, and we refer to the forthcoming Remark \ref{rem: differenza noi loro} for more details. Finally, in the last part of the proof we characterize the growth of the solution using the Alt-Caffarelli-Friedman monotonicity formula for $(\mathcal{G},h)$-equivariant solutions.

By Lemma \ref{lem: problemi di minimo}, we know that the optimal value $\ell$ (see Definition \ref{def equivariant optimal value}) is achieved by a nonnegative $(\mathcal{G},h)$-equivariant function $\varphi \in H^1(\S^{N-1},\R^k)$. Differently from the previous section, we take
\begin{equation}\label{normalization maggio}
\int_{\S^{N-1}} \varphi_i^2 = \frac{1}{k} \quad \Longleftrightarrow \quad \sum_{i=1}^k \int_{\S^{N-1}} \varphi_i^2 = 1.
\end{equation}
This choice is possible, since the minimum problem for $\ell$ is invariant under scaling of type $t \mapsto t \varphi$ with $t \in \R$, and simplifies some computation. 

\begin{lemma}\label{lem: existence}
For any $\beta >0$ there exists a $(\mathcal{G},h)$-equivariant solution $\{\mf{U}_\beta\}$ to the problem
%\begin{equation}\label{syst beta B_1}
\[	
	\begin{cases}
		\Delta U_{i,\beta} = \beta U_{i,\beta} \sum_{j \neq i} U_{j,\beta}^2 &\text{in $B_1$}\\
		U_{i,\beta}>0 & \text{in $B_1$} \\
		U_{i,\beta}= \varphi_i &\text{on $\partial B_1= \S^{N-1}$} .
	\end{cases}
\]
%\end{equation}
Moreover
\begin{itemize}
	\item[($i$)] $U_{i,\beta}(0) = U_{j,\beta}(0)$ $\forall i,j = 1, \dots, k$  and $\beta > 0$;
	\item[($ii$)] letting 
	\[
	\mathcal{E}_\beta(\mf{U}) = \int_{B_1} \sum_{i=1}^k |\nabla U_i|^2 + \beta\sum_{i<j} U_i^2 U_j^2,
	\]
	the uniform estimate $\mathcal{E}_\beta(\mf{U}_\beta) \leq \ell$ holds.
	\item[($iii$)] there exists a Lipschitz continuous function $\mf{0} \not \equiv \mf{U}_\infty$ such that, up to a subsequence, $\mf{U}_\beta \to \mf{U}_\infty$ in $\C^{0,\alpha}(B_1)$ for every $\alpha \in (0,1)$ and in $H^1_{\loc}(B_1)$.
\end{itemize}
%such that
%\[
%	u_{i,\beta} = h \circ u_{i,\beta} \quad \forall h \in H
%\]
%and
%\[
%	u_{i,\beta}= g \circ u_{j,\beta} \qquad \text{whenever $\omega_i = g.\omega_j$}.
%\]
%(as a result, the solution is invariant under the action of the group $G$)
\end{lemma}

\begin{proof}
It is not difficult to check that the functional $\mathcal{E}_\beta$ admits a minimizer $\mf{U}_\beta$ in the $H^1$-weakly closed set of the $(\mathcal{G},h)$-equivariant functions in $H^1(B_1,\R^k)$ with the prescribed boundary conditions. The fact that such minimizer solves the Euler-Lagrange equation is a consequence of Palais' principle of symmetric criticality. Property ($i$) follows straightforwardly by the equivariance (recall Remark \ref{rem: ne basta una}). Concerning property ($ii$), we introduce the $\ell$-homogeneous extension of $\varphi$, defined by 
\[
\phi(x):= |x|^\ell \varphi\left(\frac{x}{|x|}\right).
\]
By minimality $\mathcal{E}_\beta(\mf{U}_\beta) \le \mathcal{E}_\beta(\phi)$, so that it remains to check that $\mathcal{E}_\beta(\phi) \le \ell$. At first, since $\varphi_i$ is an eigenfunction of $-\Delta_\theta$ on $\{\varphi_i>0\}$ associated to the eigenvalue $\ell(\ell+N-2)$, the function $\phi_i$ is harmonic in $\{\phi_i>0\}$. Furthermore, by definition,
\[
\sum_i \int_{\pa B_1 } \phi_i^2 =1
\]
for every $i$. Therefore, using the Euler formula for homogeneous functions, we deduce that 
\begin{align*}
\mathcal{E}_\beta(\phi) &= \sum_i \int_{B_1} |\nabla \phi_i|^2 = \sum_i \int_{\{\phi_i>0\} \cap B_1} |\nabla \phi_i|^2  \\
& = \sum_i \int_{\pa B_1 \cap \{\phi_i>0\}} \phi_i \pa_{\nu} \phi_i = \ell \sum_i \int_{\pa B_1 \cap \{\phi_i>0\}} \phi_i^2 = \ell.
\end{align*}
It remains to prove ($iii$). By ($ii$) and the boundary conditions, the sequence $\{\mf{U}_\beta\}$ is bounded in $H^1(B_1)$, and hence it converges weakly to some limit $\mf{U}_\infty$. By compactness of the trace operator, $\mf{U}_\infty \not \equiv \mf{0}$. All the functions $\mf{U}_\beta$ are nonnegative, subharmonic and have the same boundary conditions, and hence by the maximum principle they are uniformly bounded in $L^\infty(B_1)$. This, as recalled in Section \ref{sec: prel}, implies the thesis.
\end{proof}

%\begin{lemma}\label{lem: properties of the sol}
%The functions in Lemma \ref{lem: existence} enjoy the following properties:
%\begin{itemize}
%	\item $u_{i,\beta}(0) = u_{j,\beta}(0)$ $\forall i,j = 1, \dots, k$  and $\beta > 0$;
%	\item the uniform estimate $J_\beta(\mbfu_\beta) \leq d$ holds ($d \geq 1$ was introduced in Definition \ref{def: min part});
%	\item there exists $\mbfu_\infty \in Lip(\bar B_1)$, $\mbfu_\infty \neq 0$ such that $\mbfu_\beta \to \mbfu_\infty$ in $\C^{0,\alpha}(\bar B_1)$.
%\end{itemize}
%\end{lemma}
%\begin{proof}
%The first propriety is a direct consequence of the symmetries. For the second one we can use as a test function $\mbfw \in Lip(B_1)$ defined as
%\begin{equation}\label{eqn: homogeneous extension}
%	\mbfw(x) = \left( |x|^d \varphi_1\left(\frac{x}{|x|}\right), |x|^d \varphi_2\left(\frac{x}{|x|}\right), \dots, |x|^d \phi_k\left(\frac{x}{|x|}\right)\right)
%\end{equation}
%and observe that 
%\[
%	J_\beta(\mbfu_\beta) \leq J_\beta(\mbfw) = d \qquad \forall \beta > 0.
%\]
%Finally by strong convergence of the trace we conclude that $\mbfu_\infty \neq 0$. \todo As $J_\beta$ is increasing in $\beta$, it is possible to show that actually $\mbfu_\infty = \mbfw$.
%\end{proof}

We plan to use the solutions of Lemma \ref{lem: existence} in order to construct entire solutions to \eqref{entire system}. Our method is based on a simple blow up argument. For a positive radius $r_\beta$ to be determined, we introduce
\[
	V_{i,\beta}(x) := \beta^{1/2} r_\beta U_{i,\beta}(r_\beta x) .
\]
By definition, $\mf{V}_\beta$ solves
\[
	\Delta V_{i,\beta} = V_{i,\beta} \sum_{j\neq i} V_{j,\beta}^2 \qquad \text{in $B_{1/r_\beta}$}.
\]
A convenient choice of $r_\beta$ is suggested by the following statement.
\begin{lemma}\label{lem: choice of r existence}
For any fixed $\beta > 1$ there exists a unique $r_\beta > 0$ such that
\[
	\int_{\partial B_1} \sum_{i=1}^{k} V_{i,\beta}^2 = 1.
\]
Moreover $r_\beta \to 0$, and consequently $B_{1/r_\beta} \to \R^N$, in the sense that for any compact $K \subset \R^N$ it results $K \Subset B_{1/r_\beta}$ provided $\beta$ is sufficiently large.
\end{lemma}
\begin{proof}
We have to find $r_\beta>0$ such that $\beta r_\beta^2 H(\mf{U}_\beta,r_\beta)=1$. The strict monotonicity of $H(\mf{U}_\beta,\cdot)$ (see Section \ref{sec: prel}) implies the strict monotonicity of the continuous function $r \mapsto \beta r^2 H(\mf{U}_\beta, r)$. By regularity, for any $\beta$ fixed
\[
	\lim_{r \to 0}   \beta r^2 H(\mf{U}_\beta, r) = \lim_{r \to 0} \beta \frac{r^2}{r^{N-1}} \int_{\partial B_{r}} \sum_{i=1}^{k} U_{i,\beta}^2 = \beta \lim_{r \to 0} r^2  \cdot \sum_{i=1}^{k} U_{i,\beta}^2 (0) =0,
\]
and by the normalization \eqref{normalization maggio} it results $\beta  H(\mf{U}_\beta, 1) =\beta>1$. This proves existence and uniqueness of $r_\beta$. If by contradiction $r_\beta \ge \bar r>0$, then by Lemma \ref{lem: existence}-($iii$) and by the monotonicity of $H(\mf{U}_\beta,\cdot)$ we would have
\[
1 = \beta r_\beta^2 H(\mf{U}_\beta,r_\beta)  \ge \beta \bar r^2 H(\mf{U}_\beta,\bar r) \ge \frac{\beta   \bar r^2}{2} \frac{1}{\bar r^{N-1}} \int_{\partial B_{\bar r}} \sum_{i=1}^{k} U_{i,\infty} \ge   \beta C,  
\]
which gives a contradiction for $\beta>1/C$. In order to bound from below the second to last term, we recall that since $\mf{0} \not \equiv \mf{U}_\infty$, we have $H(\mf{U}_\infty,r) \neq 0$ for all $0<r<1$ (see Section \ref{sec: prel}).
\end{proof}

\begin{lemma}
Up to a subsequence, $\mf{V}_\beta \to \mf{V}$ in $\C^2_\loc(\R^N)$, and $\mf{V}$ is an entire $(\mathcal{G},h)$-equivariant solution of \eqref{entire system} with $N(\mf{V},r) \le \ell$ for every $r>0$.
\end{lemma}
\begin{proof}
Since $\mathcal{E}_\beta(\mf{U}_\beta) \le \ell$ and $H(\mf{U}_\beta,1) =1$, by scaling and using the monotonicity of the Almgren quotient we have
\begin{equation}\label{stima almgren maggio}
N(\mf{V}_\beta, r) \le N\left(\mf{V}_\beta, \frac{1}{r_\beta}\right) = N(\mf{U}_\beta,1) \le \frac{\mathcal{E}(\mf{U}_\beta)}{H(\mf{U}_\beta,1)} \le \ell
\end{equation}
for every $0<r<1/r_\beta$, $\beta>0$. Let now $r>0$; then for $\beta$ sufficiently large 
\begin{align*}
	\frac{\de}{\de r} \log H(\mf{V}_\beta, r) &= \frac{2}{r} N_\beta(\mbfv_\beta,r) + \frac{2}{r^{N-1}H(\mf{V}_\beta,r)} \int_{B_r} \sum_{i<j} V_i^2 V_j^2\\  &\leq \frac{2\ell}{r} + \frac{2}{r^{N-1}H(\mf{V}_\beta,r)} \int_{B_r} \sum_{i<j} V_i^2 V_j^2.
	\end{align*}
Integrating the inequality with $r \in (1,R)$, and recalling \eqref{integrability del resto}, we infer that
\begin{equation}\label{doubling}
 \frac{H(\mf{V}_\beta, R)}{R^{2\ell}} \leq H(\mf{V}_\beta,1) e^\ell= e^\ell \quad \forall R \geq 1,
\end{equation}
independently of $\beta$. By subharmonicity and standard elliptic estimates, we deduce that $\mf{V}_\beta$ converges in $\C^2(B_R)$ to some limit $\mf{V}^R$, and since $R$ has been arbitrarily chosen, a diagonal selection gives convergence to an entire limit $\mf{V}$, which is clearly $(\mathcal{G},h)$-equivariant. Since $\mf{V}$ solves \eqref{entire system} and
\[
	\int_{\partial B_1} \sum_{i=1}^{k} V_{i,\beta}^2 = 1 \quad \text{and $V_{i,\beta}(0) = V_{j,\beta}(0)$ for all $i, j$}
\]
(see Lemmas \ref{lem: existence} and \ref{lem: choice of r existence}), all the components of $\mf{V}$ are nontrivial, and hence non-constant. 
%It remains to show that $N(\mf{V},r) \le \ell$ for every $r>0$. At first, by $\mathcal{C}^2_{\loc}$ convergence and by \eqref{stima almgren maggio}, $N(\mf{V},r) \leq 2\ell$ for every $r>0$. This implies that the left hand side in \eqref{integrability del resto} evaluated for $\mf{V}$ is also bounded, and hence
%\begin{equation}\label{resto nullo}
%\liminf_{r \to +\infty} \frac{r \int_{B_r} \sum_{i<j} V_i^2 V_j^2}{r^{N-1} H(\mf{V},r) } =0.
%\end{equation}
%Using this and the monotonicity of $N(\mf{V},\cdot)$, it is not difficult to deduce that $N(\mf{V},r) \leq \ell$ for every $r>0$.
\end{proof}

We now show that the growth rate of the solution is exactly equal to $\ell$. In light of the upper bound on the Almgren quotient proved in the previous lemma, this is a consequence of Theorem \ref{prop: Liouville symm}, which we prove below.

\begin{proof}[Proof of Theorem \ref{prop: Liouville symm}]
Let us assume by contradiction that there exists a $(\mathcal{G},h)$-equivariant solution $\mathbf{V}$ with growth rate less than $\ell - \eps$ for some $\eps > 0$. By monotonicity it results $N(\mf{V},r) \leq N(\mf{V},+\infty) \leq \ell -\eps$ for every $r>0$. We consider the blow-down sequence
\[
	\mf{V}_R(x) = \frac{1}{\sqrt{H(\mf{V}, R)}} \mf{V}(Rx).
\]
By Theorem 1.4 in \cite{SoTe}, it converges in $\C^{0,\alpha}_{\loc}(\R^N)$ to a limit $\mf{W}$, which is segregated, nonnegative, homogeneous with homogeneity degree $\delta := N(\mf{V},+\infty) \leq \ell -\eps$, and such that $\Delta W_i =0$ in $\{W_i>0\}$. The uniform convergence entails the $(\mathcal{G},h)$-equivariance, and hence the trace $\hat{\mf{w}}$ of $\mf{W}$ on the sphere $\S^{N-1}$ is an admissible competitor for $\ell$, in the sense that $\ell \le I_\infty(\hat{\mf{w}})$ ($I_\infty$ is defined in Lemma \ref{lem: problemi di minimo}). The value $I_\infty(\hat{\mf{w}})$ can be computed explicitly: indeed, by harmonicity, homogeneity and symmetry, $\hat w_i$ is an eigenfunction of the Laplace-Beltrami operator $-\Delta_\theta$ on a subdomain of $\S^{N-1}$, associated to the eigenvalue $\delta(\delta+N-2)$. This, by definition, implies that $I_\infty(\hat{\mf{w}}) = \delta < \ell$, in contradiction with the minimality of $\ell$.
\end{proof}

So far we proved the existence of a $(\mathcal{G},h)$-equivariant solution having growth rate $\ell$ in the weak sense of \eqref{growth rate}. It remains to show that the stronger condition \eqref{growth rate strong} holds. Before, we make the following remark.

\begin{remark}\label{rem: differenza noi loro}
Both Theorem \ref{thm: new existence} and \cite[Theorem 1.6]{BeTeWaWe} are based upon the same two-steps procedure: construction of solutions in balls $B_R$ of increasing radius, and passage to the limit as $R \to +\infty$. The main difference stays in the fact that while in \cite{BeTeWaWe} the authors prescribed the value of the functions on the boundary $\pa B_R$, we prescribed the value on $\pa B_1$, conveniently choosing $r_\beta$. This permits to simplify very much the proof of the convergence, since by the doubling property \eqref{doubling}, the normalization on $\pa B_1$ is enough to have $\mathcal{C}^2_{\loc}(\R^N)$ convergence of our approximating sequence. In \cite[page 123]{BeTeWaWe}, such compactness is proved in a different way, using fine tools such as Proposition 5.7 therein, which seems difficult to generalize in higher dimension.
\end{remark}

\begin{lemma}
It holds
\[
	\lim_{r \to \infty}\frac{1}{r^{2\ell}} H(\mf{V},r) \in (0,+\infty).
\]
\end{lemma}
\begin{proof}
It is easy to prove that the limit exists and it is less than 1. Indeed
\[
	\frac{d}{dr} \log \frac{H(\mf{V},r)}{r^{2\ell}} = \frac{H'(\mf{V},r)}{H(\mf{V},r)} - \frac{2\ell}{r} = \frac{2}{r} (N(\mf{V},r) - \ell) \leq 0,
\]
and by construction $H(\mf{V},1) = 1$. Letting
\[
	L = \lim_{r \to \infty}\frac{H(\mf{V},r)}{r^{2\ell}} 
\]
we are left to show that $L >0$. 
%Let 
%\[
%E(\mf{V},r) := E(\mf{V},r) + r^{2-N} \int_{B_r} \sum_{i<j} V_i^2 V_j^2.
%\]
Recalling that $N(\mf{V},+\infty)= \ell$, we have
\[
	L   = \lim_{r \to \infty} \left( \frac{ E(\mf{V},r) }{r^{2\ell}}\right) \cdot \lim_{r \to +\infty} \frac{H(\mf{V},r)}{ E(\mf{V},r)   } 
	\ge \frac{1}{\ell} \liminf_{r \to \infty}\frac{E(\mf{V},r)}{r^{2\ell}},
\]
and the thesis follows if 
\[
	\liminf_{r \to \infty}\frac{ E(\mf{V},r) + H(\mf{V},r)}{r^{2\ell}} > 0.
\]
To this aim, we note that with computations analogue to those in \cite[Conclusion of the proof of Theorem 1.5]{SoZi2} we can prove that
\[
	\frac{ E(\mf{V},r) + H(\mf{V},r)}{r^{2\ell}} \geq \frac{C}{r^{2\ell}} \left( J_1(r) \dots J_k(r) \right)^{1/k} = C\left(  \frac{1}{r^{2\ell k}}J_1(r) \dots J_k(r) \right)^{1/k},
\]
where the integrals $J_i$ are evaluated for the function $\mf{V}$. Since $\mf{V}$ is a $(\mathcal{G},h)$-equivariant solution of \eqref{entire system}, we are in position to apply the Alt-Caffarelli-Friedman monotonicity formula of Proposition \ref{prop: acf equiv}, whence
\[
	\frac{  E(\mf{V},r) + H(\mf{V},r)}{r^{2\ell}} \geq C \left(  J_1(1) \dots J_k(1) \right)^{1/k}e^{Cr^{-1/2}} \ge C e^{Cr^{-1/2}}  \]
	for every $r>1$.
\end{proof}

%\bibliography{bibliography}
%\bibliographystyle{abbrv}

\end{document}